\documentclass{article}
\usepackage{graphicx} % Required for inserting images

\usepackage{comment}
\usepackage{float}
\usepackage{amsmath}
\usepackage{listings}
\usepackage{amsfonts}
\usepackage{pgf,tikz,pgfplots}
\pgfplotsset{compat=1.15}
\usepackage{mathrsfs}
\usepackage{comment}
\usetikzlibrary{arrows}
\usepackage[normalem]{ulem}
\usepackage{amsmath,amssymb,amsthm}
\usepackage{subcaption} 
\usepackage{hyperref}

\usepackage{cite}
\usepackage{authblk}

\newtheorem{theorem}{Theorem}
\newtheorem{proposition}{Proposition}
{\theoremstyle{remark} }

\title{Solutions of the Riemann problem for the Self-Organized
  Hydrodynamics model }

\author[1]{Christophe Berthon}
\author[1]{Marie Compain}
\author[2]{Carlos Par\'es}
\author[3]{Ernesto Pimentel-Garc{\'\i}a}

\affil[1]{Laboratoire de Math\'ematiques Jean Leray, CNRS UMR 6629, Nantes Universit\'e, 2 rue de la Houssini\`ere, BP 92208, 44322 Nantes, France}
\affil[2]{Departamento de An\'alisis Matem\'atico, Estad{\'\i}stica e Investigaci\'on Operativa, y Matem\'atica Aplicada, Universidad de M\'alaga, Campus de Teatinos s/n, M\'alaga 29080, Spain}
\affil[3]{Departamento de Matem\'atica Aplicada, Universidad de M\'alaga, Campus de Teatinos s/n, M\'alaga 29080, Spain}

\date{\today}

\begin{document}

\maketitle

\begin{abstract}
The present paper is devoted to exhibit solutions of the Riemann
problem for the Self-Organized Hydrodynamics model. The $2\times2$
system under consideration is strictly hyperbolic. With ambiguous non-conservative products, we adopt
a non usual Generalized Rankine-Hugoniot jump condition to suitably
define the shock wave solutions. This jump condition is obtained from the information given by the travelling wave solutions of the viscous extension of the system that involves the diffusion terms arising from the physics of the problem. We are then able to determine the
simple waves of the non-conservative system studied here and we are
able to construct some solutions of the Riemann problem of interest.
\end{abstract}

%%%%%%%%%%%%%%%%%%%%%%%%%%%%%%%%%%%%%%%%%%%%%%%%%%%%%%%%%%%
%%%%%%%%%%%%%%%%%%%%%%%%%%%%%%%%%%%%%%%%%%%%%%%%%%%%%%%%%%%
%%%%%%%%%%%%%%%%%%%%%%%%%%%%%%%%%%%%%%%%%%%%%%%%%%%%%%%%%%%

\section{Introduction}
In the present work, we are concerned by an analysis of the Riemann
problem for a PDE system which models a self-organized
hydodynamics. This system has been derived in
\cite{Degond2012,Degond2008} as the hydrodynamic limit of the
well-known Cousin-Vicsek model decribed in \cite{vicsek1995novel}. The
resulting set of PDE models the behavior of a collective motions of
self-propelled particules subject to alignment interactions. These
models have a wide range of applications. For example, they have been
used in biology to study the collective behavior of cells, including
cell migration, tissue formation and the propagation of biochemical
signals \cite{Creppy2016}. These models have also been adopted in
crowd dynamics, where they have inspired the simulation of human
pedestrian flows capable of self-organising to navigate complex
spaces, manage congestion, and optimise evacuation paths
\cite{Degond2013}. These models have also been applied to opinion
dynamics, where they allow for the continuum modeling of an intention
field subject to weak social interactions and consensus processes. In
the field of social self-organisation, these models are used to
capture the macroscopic evolution of collective choices and human
interactions from microscopic alignment rules \cite{Degond2017a}.
These models can be applied in robotics, where they have inspired the
design of autonomous robot systems capable of self-organising to
accomplish collective tasks, such as exploring the environment or
forming complex structures \cite{Degond2014}. In the field of
nanotechnology, these models are used to design nano-robot systems
capable of moving and organising themselves or to applied to 3D
autonomous drone fleets \cite{Degond2025,Degond2017b}.  They are also
found in ecology, where they can be used to study the collective
behavior of living organisms in aquatic ecosystems, such as schools of
fish or colonies of plankton \cite{Degond2017}.

It is worth noticing that these Self-Organized Hydrodynamics models
(SOH model) are essential when considering complex hydrodynamic
phenomena emerging from the interactions between a large number of
particles following simple rules. The model under consideration in
the present paper is proposed in \cite{Degond2008} (see also
\cite{Degond2012,motsch2011numerical,Frouvelle2012}) and, adopting a
one dimentional flow, it reads for all $x\in{\mathbb R}$ and $t>0$
\begin{equation}\label{sys0}
\left\{
\begin{array}{l}
\displaystyle \partial_t \rho + \partial_x (\rho u) = 0,\\
\displaystyle  \rho(\partial_t u + c u \partial_x u) + \lambda (1 -u^2) \partial_x \rho = 0, \\
\displaystyle  \rho(\partial_t v + c u \partial_x v) - \lambda uv \partial_x \rho = 0,
\end{array}
\right.
\end{equation}
where $\rho(x,t)>0$ stands for the density and the pair
$(u(x,t),v(x,t))$ is the velocity vector. Here, $c\in(0,1)$ and
$\lambda>0$ define constants to be prescribed by the physics.

A specificity of the Cousin-Vicsek model stays in the velocity
definition since the self-propulsion speed is assumed to be constant
and identical for all the particles. As a consequence, we impose
\begin{equation}\label{constraint}
u^2 + v^2 = 1,
\end{equation}
which makes one of the velocity equation redundant in \eqref{sys0}. It
is thus clear that, in the model adopted here, the velocity is only
governed by its direction. After
\cite{motsch2011numerical,compain:hal-05490904}, considering
$u=\cos(\theta)$ and $v=\sin(\theta)$ with $\theta\in (-\pi,\pi]$ the
  direction of the flow, the system made of \eqref{sys0} and
  \eqref{constraint} is, at least for smooth solutions, easily shown to
  be equivalent to
\begin{equation}\label{pde}
\left\{
\begin{array}{l}
\displaystyle \partial_t \rho + \partial_x (\rho \cos(\theta)) = 0,\\
\displaystyle \partial_t \theta + \partial_x (c \sin(\theta)) - \lambda \frac{\sin(\theta)}{\rho} \partial_x \rho = 0. 
\end{array}
\right.
\end{equation}

First of all, we emphasize that the above system is in
non-conservation form. This nonlinear system turns out to be hyperbolic (see
\cite{motsch2011numerical}). Hence, discontinuities may arise in a
finite time so that the system \eqref{pde} may not admit smooth
solutions. Research on non-conservative hyperbolic systems has produced
important results during the last three decades.  Following the
pioneering definition introduced by Dal Maso, LeFloch, and Murat
\cite{DMLFM1995}, a complete theory, called path-theory, has emerged that encompasses weak
solution formulations
\cite{volpert1967,lefloch1989,DMLFM1995,lefloch2004}, existence
results for Riemann \cite{lefloch1989,DMLFM1995,LeFloch2007,Thanh2009} and Cauchy problems
\cite{lefloch1989,crasta2002}, uniqueness in the space of bounded
variation functions \cite{baiti2001}, derivation and convergence of
finite volume approximations
\cite{castro2008,pares2006}. Alongside these theoretical foundations,
a wide range of continuum mechanics models have been systematically
studied due to their importance in multiphase flows and turbulence
modeling
\cite{audebert2006,berthon2007,Berthon2008,chalons2006,chalons2007} or
multilayer shallow equation
\cite{pimentelgarcia2022cell,Castro-Diaz2011,Fernandez-Nieto2011,Castro2012,CastroDiaz2013,Audusse2005}.

Despite the fundamental interest of the path-theory devolopped by Dal
Maso, LeFloch, and Murat \cite{DMLFM1995}, this approach needs a
choice in order to close the definition of the non-conservative
products within a shock discontinuity wave. In general, this choice is
made to simplify the computations associated to the path theory. In
some particular frameworks, it is possible to relevantly substitute
the path-theory choice by physical arguments. After the work by
Sainsaulieu \cite{Sainsaulieu1991}, adopting suitable viscosity is a
good ingredient to define shock waves (for instance, see
\cite{Raviart1995,berthon2007,chalons2006,Berthon2002,Berthon2008}). Here,
considering system \eqref{sys0}-\eqref{constraint}, a suitable
viscosity has been derived in \cite{Degond2012} to get the following
viscous system:
\begin{equation}
\left\{
\begin{array}{l}
\displaystyle \partial_t \rho + \partial_x (\rho u) = 0,\\
\displaystyle  \rho(\partial_t u + c u \partial_x u) + \lambda (1 -u^2) \partial_x \rho = 
c_v\left((1 - u^2) \partial_{xx} (\rho u) - uv \partial_{xx}(\rho v) \right), \\
\displaystyle  \rho(\partial_t v + c u \partial_x v) - \lambda uv \partial_x \rho = 
c_v\left(-uv\partial_{xx} (\rho u) + (1- v^2) \partial_{xx}(\rho v) \right), \\
u^2 + v^2 = 1,
\end{array}
\right.
\end{equation}
where $c_v$ is a given viscous coefficient prescribed by the
physics. The above system reformulates as follows as soon as the
velocity direction is only considered:
\begin{equation}\label{sys_visquous}
\left\{
\begin{array}{l}
\displaystyle \partial_t \rho + \partial_x (\rho \cos(\theta)) = 0, \\
\displaystyle \partial_t \theta + \partial_x (c \sin(\theta)) - \lambda \frac{\sin(\theta)}{\rho} \partial_x \rho =
c_v \left(\frac{2}{\rho} \partial_x \theta \partial_x \rho + \partial_{xx} \theta \right).
\end{array}
\right.
\end{equation}
Arguing this viscous extension of the initial system and according to
\cite{Sainsaulieu1991}, in \cite{compain:hal-05490904} the authors
involve travelling wave solutions to get a non usual Rankine-Hugoniot
relation. Moreover, they show that the system \eqref{pde} completed
by the non standard Rankine-Hugoniot relation, called Generalized
Rankine-Hugoniot condition and detailed later, is equivalent to \eqref{sys0}-\eqref{constraint}
for weak solutions eventually containing discontinuous shock waves.

In \cite{motsch2011numerical} it has been shown that System \eqref{pde}
can be written in conservation form as follows:
\begin{equation}\label{scl}
\displaystyle \partial_t \left( \begin{array}{c} \rho \\ f_1(\theta) \end{array} \right)
+  \partial_x \left( \begin{array}{c} \rho \cos(\theta)) \\
 c f_2(\theta) - \lambda \log(\rho) \end{array} \right) = 0 ,
\end{equation}
with
$$
f_1(\theta) = \log\left| \tan \left(\frac{\theta}{2} \right)\right|, \quad 
f_2(\theta) = \log |\sin (\theta)|.$$
Systems \eqref{pde} and \eqref{scl} are equivalent for smooth solutions but, as it will be seen, the standard Rankine-Hugoniot condition for \eqref{scl} 
\begin{equation}\label{RH_con}
\left\{
\begin{array}{l}
\sigma(\rho_R - \rho_L) = \rho_R \cos(\theta_R) - \rho_L \cos(\theta_L),\\
\smallskip
\displaystyle\sigma(f_1(\theta_R) - f_1(\theta_L)) = cf_2(\theta_R) - cf_2(\theta_L)
+ \lambda \log(\rho_R)) - \lambda \log(\rho_L),
\end{array}
\right.
\end{equation}
is not equivalent to the Generalized Rankine-Hugoniot condition found in \cite{compain:hal-05490904}. Therefore, the concept of weak solution is different.
Although using a conservative formulation makes it easier to solve the problem numerically, here we opt for the non-conservative formulation \eqref{pde} accompanied by the Generalized Rankine-Hugoniot condition in \cite{compain:hal-05490904} as it is the formal limit of \eqref{sys_visquous} as $c_v \to 0$. Since \eqref{sys_visquous} involves the expression of the viscous term coming from the physics of the problem, the shock waves satisfying the Generalized Rankine-Hugoniot condition are expected to capture better the microscale diffusive effects.

In \cite{motsch2011numerical} (see also
\cite{compain:hal-05490904}) the details of the algebra of the system
\eqref{pde} are given. As a consequence, once the jump conditions has been chosen, all the needed ingredients are
now known to exhibit the solutions of the Riemann problem associated
to \eqref{pde}. For instance, the reader is referred to
\cite{LeFloch2007,Thanh2009} where Riemann solutions are derived for
non-conservative hyperbolic systems.  To address such an objective, the
paper is organized as follows. In the next section, we recall the
algebra of the system and we show that \eqref{pde} is strictly
hyperbolic and we present the nature of the fields (see
\cite{motsch2011numerical}). Next, in Section 3, we detail the simple
waves; namely both rarefaction and shock waves. From now on, we
underline that the Generalized Rankine-Hugoniot conditions established
in \cite{compain:hal-05490904} are here adopted. Equipped with the
simple waves, we are able to derive the solutions of the Riemann
problem in Section 4. Let us note that the flow angles $\theta$ are
imposed to belong to a particular domain within $(-\pi,\pi]$ because
of a possible change in the nature of the fields.

%%%%%%%%%%%%%%%%%%%%%%%%%%%%%%%%%%%%%%%%%%%%%%%%%%%%%%%%%%%
%%%%%%%%%%%%%%%%%%%%%%%%%%%%%%%%%%%%%%%%%%%%%%%%%%%%%%%%%%%
%%%%%%%%%%%%%%%%%%%%%%%%%%%%%%%%%%%%%%%%%%%%%%%%%%%%%%%%%%%

\section{Main nature of the SOH model}
For the sake of completeness, we now present the main results
concerning the algebra of the evolution law \eqref{pde}. The
computations presented in this section were performed in
\cite{motsch2011numerical} and they are essential in the sequel of
this paper. In order to clarify in the notations, we set
\begin{equation}\label{U}
U = \left[ \begin{array}{c} \rho \\ \theta \end{array} \right] \in [0,\infty)\times(-\pi, \pi],
\end{equation}
to write the system \eqref{pde} in a quasi-linear system as follows:
$$
U_t + A(U)U_x = 0$$
where we have introduced
\begin{equation}\label{A}
A(U) = \left[ \begin{array}{cc} \cos \theta & -\rho \sin \theta \\ \displaystyle- \lambda \frac{\sin \theta}{\rho} & c \cos \theta
\end{array} \right].
\end{equation}

The above matrix admits two real eigenvalues given by
\begin{equation}\label{av1}
\mu_\pm(\theta) = \frac{c + 1}{2} \cos \theta \pm \gamma, \quad \gamma = \sqrt{\frac{(c-1)^2}{4} \cos^2 \theta + \lambda \sin^2 \theta}.
\end{equation}
Since $c<1$, the eigenvalues are distinct and then $A(U)$ is
diagonalizable in ${\mathbb R}$. As a consequence, the system
\eqref{pde} is strictly hyperbolic in $(0, \infty) \times (-\pi, \pi]$. Moreover, let us note that the
associated eigenvectors read
$$
R_- (\theta) = \left[ \begin{array}{c} \displaystyle c \cos \theta - \mu_-(\theta)\\  \frac{\lambda}{\rho} \sin
    \theta\end{array}
  \right] 
\quad \text{ and } \quad 
R_+ (\theta) = \left[ \begin{array}{c} \rho \sin
    \theta\\ \displaystyle \cos \theta - \mu_+(\theta) \end{array}
  \right].
$$

The characteristic field $(\mu_+,R_+)$ is neither genuinely nonlinear
nor linearly degenerated since $\nabla_U\mu_+\cdot R_+=0$ as soon as
$\theta=\pm\theta^\star$ with
$$
\theta^* = \arctan\left(\sqrt{\frac{4\lambda^2 - (2\lambda + c)(c-1)^2 }{\lambda (c +1)^2}}\right).
$$
Let us remark that the physical values of $c$ and $\lambda$ enforce
$4\lambda^2 - (2\lambda + c)(c-1)^2>0$ so that
$\theta^\star\in(-\pi,\pi]$ exists. In addition, we notice that
  $\mu_+(\theta)$ reaches its maximum for $\pm\theta^\star$ and its
minimum for $\theta=\pi$. It has also a local minimum at $\theta = 0$.

Likewise, the characteristic field $(\mu_-,R_-)$ is neither genuinely
nonlinear nor linearly degenerated because $\nabla_U\mu_-\cdot R_-$
vanishes for $\theta=\pm(\theta^\star-\pi)$. Moreover, we notice that
the minimum value of $\mu_-(\theta)$ is reached for
$\theta=\pm(\theta^\star-\pi)$ while its maximum value is obtained
for $\theta=0$. It has also a local maximum in $\theta = \pi$.

For the sake of simplicity in the notations, we set
$$
\overline{\mu}_-= \mu_-(\pi-\theta^*) = \min(\mu_-(\theta)), \quad \overline{\mu}_+= \mu_+(\theta^*) = \max(\mu_+(\theta)).
$$

To illustrate these behaviors, in Figure \ref{fig:eigenvalues} we display the graphs of the
eigenvalues and the values of $\theta$ corresponding to the change of
nature of the characteristic fields; formally the values of $\theta$
where the genuinely nonlinear character is lost.

\begin{figure}[htpb]
\begin{center}
\includegraphics[width=0.6\linewidth]{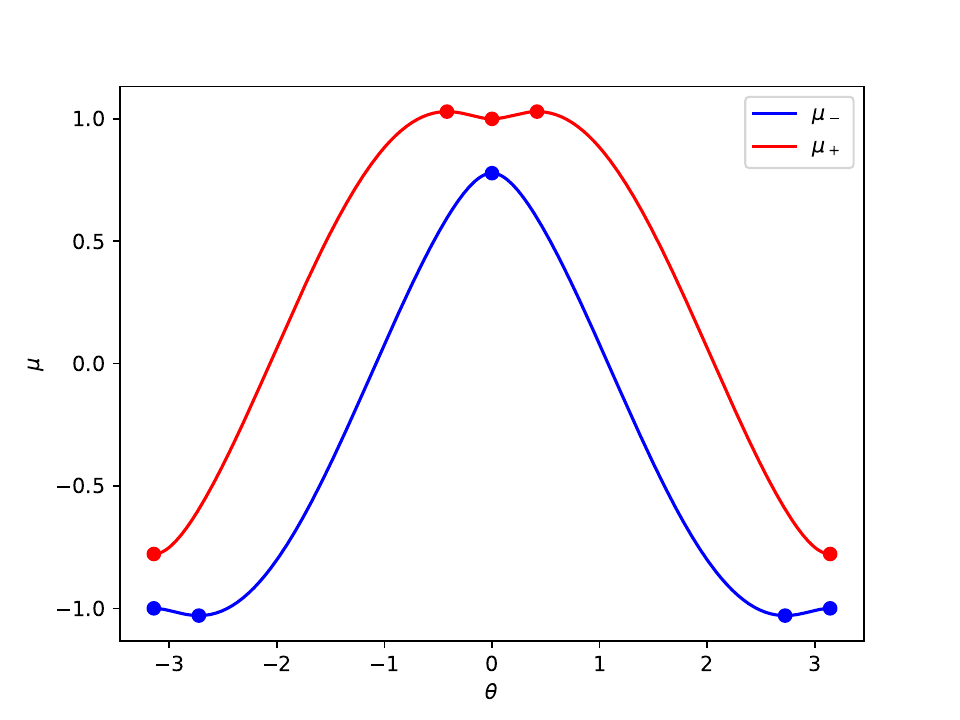}
\caption{Graphs of the eigenvalues $\mu_\pm(\theta)$. The angles $\pm\pi$, $0$, $\theta^* - \pi$, $ \pi-\theta^*$ (resp. $\pm\pi$, $0$, $\pm\theta^*$) in which the first (resp. second) characteristic field is not genuinely nonlinear,  are marked with a green circle (resp. a red star) in the corresponding graph.}
\label{fig:eigenvalues}
\end{center}
\end{figure}

Because of the complexity of the involved algebra calculus, in the
present work we decide to avoid eventual change in the nature of the
characteristic fields. As a consequence, we focus here in the analysis
of system \eqref{pde} in the region
\begin{equation}\label{defOmega} 
\Omega = (0, \infty) \times (\theta^*, \pi - \theta^*),
\end{equation}
where
$\nabla_U\mu_\pm\cdot R_\pm\not=0$. With some abuse in the definition,
we then enforce our study in regions where the characteristic fields
are genuinely nonlinear. However, it can be easily checked that, if $U(x,t) =
[\rho(x,t), \theta(x,t)]^T$ is a solution of \eqref{pde} such that
$U(x,t) \in \Omega$ for all $x$, $t$, then $\tilde U(x,t) =
[\rho(-x,t), -\theta(-x,t)]^T$ is a solution satisfying $\widetilde
U(x,t) \in \widetilde \Omega: = (0, \infty) \times (\theta^*-\pi,
-\theta^*) $ so that all the conclusions of our analysis also applies to
$\widetilde \Omega$.

Since it will be usefull in the sequel, this section is concluded by
the definition of the inverse function of $\mu_\pm$. We then define the functions
\begin{equation}\label{invav}
\chi_\pm(x) = \arccos\left(\frac{x(c+1) \pm \sqrt{x^2(c-1)^2 + 4\lambda(c + \lambda - x^2)}}{2(c+ \lambda)}\right).
\end{equation}
Hence, it can be checked that
\begin{itemize}
 \item $\chi_+:[\mu_-(\pi-\theta^*), \mu_-(0)] \to \mathbb{R}$ is the inverse function of $\mu_-$ in the interval $[0, \pi - \theta^*]$;
\item $\chi_-:[\mu_+(\pi), \mu_+(\theta^*)] \to \mathbb{R}$ is the inverse function of $\mu_+$ in the interval $[\theta^*, \pi]$.
\end{itemize}

%%%%%%%%%%%%%%%%%%%%%%%%%%%%%%%%%%%%%%%%%%%%%%%%%%%%%%%%%%%
%%%%%%%%%%%%%%%%%%%%%%%%%%%%%%%%%%%%%%%%%%%%%%%%%%%%%%%%%%%
%%%%%%%%%%%%%%%%%%%%%%%%%%%%%%%%%%%%%%%%%%%%%%%%%%%%%%%%%%%

\section{Simple curves}
Considering strictly hyperbolic systems, the Riemann problem is resolved by
connecting constant states through elementary simple waves. For
genuinely nonlinear characteristic fields, these transitions take the
form of either smooth rarefaction waves or discontinuous shock
waves. Continuous rarefaction waves come from the integral curves of the characteristic fields and
Riemann Invariant whereas shock waves represent sharp jumps
constrained by the Rankine-Hugoniot relations. Because of the
non-conservative form of the system \eqref{pde}, the expected
Rankine-Hugoniot conditions are non standard and they will be given
according to \cite{compain:hal-05490904}. Together, these two
fundamental waves describe the complete dynamic structure of the
solution of the Riemann problem of interest.

From now on, we emphasize that the characteristic fields are enforced to be genuinely
nonlinear so that we restrict the present analysis within the domain
$\Omega$ defined by \eqref{defOmega}.

\subsection{Riemann invariants and rarefaction waves}
After \cite{motsch2011numerical}, the integral curves of the characteristic fields starting from $U_L = [\rho_L, \theta_L]^T$ are given by
$$
\rho = \rho_L\exp\left( \int_{\theta_L}^\theta \frac{\sin s }{\cos s - \mu_\pm(s)}\, ds\right).
$$
Let us set
\begin{equation}\label{charcurv}
F_\pm(\theta) = \int^\theta \frac{\sin s }{\cos s - \mu_\pm(s)}\, ds,
\end{equation}
so that 
$$
\rho = \rho_L\exp\left(F_\pm(\theta)-F_\pm(\theta_L)\right),
$$
and then the associated Riemann invariants read
\begin{equation}\label{Rinv}
I_\pm = \log(\rho) - F_\pm(\theta),
\end{equation}
The primitive function $F_\pm(\theta)$ can be computed with elementary
functions and its expression is shown in Appendix \ref{appendix_1}.

Next, given a state $U_L = [\rho_L, \theta_l]^T \in \Omega$, the curve
of states that can be linked to $U_L$ through a 1-rarefaction curve is
defined as follows:
\begin{equation}\label{R-}
\mathcal{R}_-(U_L) = \left\{ \left[\begin{array}{c} \rho_L\exp(F_-(\theta_R)-F_-(\theta_L)) \\ \theta_R \end{array} \right], \quad \theta_R \in [\theta^*, \theta_L]\right\}.
\end{equation}
Observe that, since $\mu_-(\theta)$ is a decreasing function for $\theta \in [\theta^*, \pi-\theta^*]$, the rarefaction wave can only be built if $\theta_R < \theta_L$.
The analytical expression of the rarefaction wave linking $U_L$ with
$U_R \in \mathcal{R}_-(U_L)$ is obtained by parameterizing the Riemann
invariant as a function of the eigenvalue $\mu_-$. To achieve such
reformulation, it is sufficient to use the adequate inverse of the
eigenvalues according to \eqref{invav} to get
$$
U(x,t) = \begin{cases}
U_L & \text{if $x \leq \mu_-(U_L)t$;}\\
\left[\begin{array}{c} \rho_L\exp(F_-(\chi_+(\frac{x}{t}))-F_-(\theta_L)) \\ \chi_+(\frac{x}{t}) \end{array} \right]
& \text{if $\mu_-(U_L)t < x < \mu_-(U_R)t$;}\\
U_R & \text{if $x \geq \mu_-(U_R)t$.}
\end{cases}
$$

Similarly, the expression of the curves composed by the states that
can be linked to $U_L$ through a 2-rarefaction wave is now defined by
\begin{equation}\label{R+}
\mathcal{R}_+(U_L) = \left\{ \left[\begin{array}{c} \rho_L\exp(F_+(\theta_R)-F_+(\theta_L)) \\ \theta_R \end{array} \right], \quad \theta_R \in [\theta^*, \theta_L]\right\},
\end{equation}
and the analytical expression of the rarefaction wave linking $U_L$ with $U_R \in \mathcal{R}_+(U_L)$ is given then by
$$
U(x,t) = \begin{cases}
U_L & \text{if $x \leq \mu_+(U_L)t$;}\\
\left[\begin{array}{c} \rho_L\exp(F_+(\chi_-(\frac{x}{t}))-F_+(\theta_L)) \\ \chi_-(\frac{x}{t}) \end{array} \right]
& \text{if $\mu_+(U_L)t < x < \mu_+(U_R)t$;}\\
U_R & \text{if $x \geq \mu_+(U_R)t$.}
\end{cases}
$$
These curves $\mathcal{R}_\pm(U_L)$ are illustrated in Figure \ref{fig:Simple_curves}.

\subsection{Jump conditions}
Since the system \eqref{pde} is in non-conservative form, it is well-known that the jump conditions are ambiguous \cite{DMLFM1995}. According to \cite{Sainsaulieu1991}, we adopt here a viscous strategy to suitably define the required Rankine-Hugoniot relations. To address such an issue, the admissible discontinuities are characterized as the limits of the traveling waves of the viscous problem \eqref{sys_visquous} as $c_v$ tends to 0. In \cite{compain:hal-05490904}, the viscous profiles were computed and the following jump condition was derived: a pair of state $U_L = [\rho_L, \theta_L]^T$, $U_R = [\rho_R, \theta_R]^T$ defines a shock wave of speed $\sigma$ compatible with the viscous model \eqref{sys_visquous} as soon as the following conditions hold:
\begin {itemize}
\item[(SW1)] The states $U_L$, $U_R$, and the velocity $\sigma$ satisfy
\begin{equation}\label{RH}
\left\{
\begin{array}{l}
\sigma(\rho_R - \rho_L) = \rho_R \cos(\theta_R) - \rho_L \cos(\theta_L); \\
G_\sigma (\theta_R) - G_\sigma (\theta_L) = 0;
\end{array}
\right.
\end{equation}
where $G_\sigma$ is given for $\cos(\theta) \not= \sigma$ by
\begin{equation} \label{G_sigma}
G_\sigma'(\theta) = \frac{- \sigma + c \cos(\theta) - \frac{\lambda \sin(\theta)^2}{\cos(\theta) - \sigma}}{(\cos(\theta) - \sigma)^2}.
\end{equation}

\item[(SW2)] If $\theta_L < \theta_R$, for all $\theta \in (\theta_L, \theta_R)$, we have $ G_\sigma(\theta) - G_\sigma(\theta_L) >0 $,\\
if $\theta_L > \theta_R$, for all $\theta \in (\theta_R, \theta_L)$, we have $ G_\sigma(\theta) - G_\sigma(\theta_L) < 0$.

\item[(SW3)] For all $\theta$ between $\theta_L$ and $\theta_R$, we have $\cos(\theta) \not= \sigma$.
\end{itemize}

The function $G_\sigma(\theta)$ can be computed with elementary functions and its expression is detailed in Appendix \ref{appendix_2}.

At this level, it is worth noticing that this very nonstandard jump condition is a remarkable counterexample to a possible equivalence of the evanescent viscosity approach \cite{Sainsaulieu1991} and the path-theory \cite{DMLFM1995}. Indeed, after the work by LeFloch \cite{DMLFM1995,lefloch1989,berthon2012}, a sense to the non-conservative product as a Borel measure is given on the basis of the choice of a family of paths
$\Phi: [0,1] \times \Omega \times \Omega \to \Omega$, i.e. a Lispchitz-continuous map verifying 
$$
\Phi(0; U_L, U_R) = U_L, \quad \Phi(1; U_L, U_R) = U_R, $$
for all $U_L, U_R \in \Omega$, and
$$
\Phi(s; U, U) = U,$$
for every $U \in \Omega$, $s \in [0,1]$. Once the  family of paths has been chosen, a pair of states $U_L = [\rho_L, \theta_L]^T$, $U_R = [\rho_R, \theta_R]^T$ defines a shock wave of speed $\sigma$ compatible with the chosen family if it satisfies
\begin{equation}\label{RH2}
\sigma(U_R - U_L ) = \int_0^1 A(\Phi(s; U_L, U_R))\partial_s \Phi(s; U_L, U_R) \, ds.
\end{equation}

Now, it is worth asking whether it is possible to choose a family of paths $\Phi$ such that jump conditions \eqref{RH} and \eqref{RH2} are equivalent. The following statement gives a negative answer:
\begin{proposition} \label{prop:nopath}
It is not possible to find a family of paths that depends continuously on the parameter $\lambda \in [0, \infty)$ such that the jump condition \eqref{RH2} is equivalent to \eqref{RH} for all $\lambda$.
\end{proposition}
\begin{proof}
Let us assume that there exists a family of paths such that both jump relations \eqref{RH} and \eqref{RH2} are equivalent for all $\lambda$. It can be easily checked that
jump condition \eqref{RH2} reads 
\begin{equation}\label{RH2bisphi}
\left\{
\begin{array}{l}
\sigma(\rho_R - \rho_L) = \rho_R \cos(\theta_R) - \rho_L \cos(\theta_L); \\
\smallskip
\displaystyle\sigma(\theta_R - \theta_L) = c\sin(\theta_R) - c\sin(\theta_L) - \lambda \int_0^1 \frac{\sin\left(\Phi_\theta(s)\right)}{\Phi_\rho(s)} \partial_s \Phi_p(s)\,ds.
\end{array}
\right.
\end{equation}
with $\Phi_\rho$ and $\Phi_\theta$ represent the components of $\Phi$ where the dependence on $U_L$, $U_R$ of the paths have been dropped for shortness. On the other hand, \eqref{RH} writes
\begin{equation}\label{RH2bisG}
\left\{
\begin{array}{l}
\sigma(\rho_R - \rho_L) = \rho_R \cos(\theta_R) - \rho_L \cos(\theta_L); \\
\smallskip
\displaystyle\sigma(\theta_R - \theta_L) = c\sin(\theta_R) - c\sin(\theta_L) - \mathcal{G}_\sigma(\theta_R) + \mathcal{G}_\sigma(\theta_L),
\end{array}
\right.
\end{equation}
where  $\mathcal{G}_\sigma$ satisfies
$$
\mathcal{G}'_\sigma(\theta) = \frac{(\sigma - c \cos(\theta))(1  -(\cos(\theta) - \sigma)^2) +  \frac{\lambda \sin(\theta)^2}{\cos(\theta) - \sigma}}{(\cos(\theta) - \sigma)^2}.$$
Therefore, \eqref{RH} and \eqref{RH2} are equivalent if and only if
$$
\lambda \int_0^1 \frac{\sin\left(\Phi_\theta(s)\right)}{\Phi_\rho(s)} \partial_s \Phi_p(s)\,ds
= \mathcal{G}_\sigma(\theta_R) - \mathcal{G}_\sigma(\theta_L),
$$
for all pairs of states $(U_L, U_R)$ that can be linked by an admissible shock. In particular, in the limit of $\lambda$ to 0, the above relation is satisfied if and only if 
\begin{equation}\label{RHconstraint}
\int_{\theta_L}^{\theta_R} \frac{(- \sigma + c \cos(\theta))(1  - (\cos(\theta) - \sigma)^2)}{(\cos(\theta) - \sigma)^2} = 0,
\end{equation}
as long as the triple $(U_L,U_R;\sigma)$ satisfies the jump condition \eqref{RH2bisphi} associated with $\lambda = 0$ as follows:
\begin{equation}\label{RHst}
\left\{
\begin{array}{l}
\sigma(\rho_R - \rho_L) = \rho_R \cos(\theta_R) - \rho_L \cos(\theta_L),\\
\smallskip
\displaystyle\sigma(\theta_R - \theta_L) = c\sin(\theta_R) - c\sin(\theta_L).
\end{array}
\right.
\end{equation}

However, in general it is not possible to verify simultaneously \eqref{RHconstraint} and \eqref{RHst}. For instance, with the following values of the left and right states and the shock velocity: 
$$
U_L = [1,1]^T, U_R = [12.804708058241372, 1.1]^T, \sigma = 0.19178910431251792$$
the jump relation \eqref{RHst} is satisfied for $c = 0.5$, but the condition \eqref{RHconstraint} hold false:
$$
\int_{\theta_L}^{\theta_R} \frac{(- \sigma + c \cos(\theta))(1  - (\cos(\theta) - \sigma)^2)}{(\cos(\theta) - \sigma)^2} = 2.027069051039645.
$$
As a consequence, both jump conditions \eqref{RH2} and \eqref{RH} are not equivalent for $\lambda$ to $0$. By continuity, it is also possible to find pair of states that can be linked by an admissible shock wave for which 1 and 2 do not coincide for sufficiently small positive $\lambda$. The proof is thus completed.
\end{proof}

% \textcolor{red}{
% \begin{remark}
% Observe that, for $\lambda = 0$, System \eqref{pde} becomes a system of conservation laws with numerical flux
% $$
% F(U) = \left[ \begin{array}{c} \rho \cos(\theta))  \\ c \sin(\theta) \end{array} \right]. $$
% According to the proof of Proposition \ref{prop:nopath}, the standard RH condition for this system \eqref{RHst} is not equivalent to the jump condition \eqref{RH} with $\lambda = 0$. Although we are interested in problems with $\lambda > 0$, this fact put on evidence the non-standard character of the jump conditions \eqref{RH}.
% \end{remark}
% }

To conclude this important comment about non-equivalence between path-theory and vanishing viscosity strategy, let us emphasize that the system \eqref{pde} becomes conservative as soon as we enforce $\lambda=0$. As a consequence, \eqref{pde} comes with usual Rankine-Hugoniot conditions easily given by \eqref{RHst}. Next, the path-theory recovers the standard Rankine-Hugoniot relation for system in conservation form, so that it is thus natural to get the standard Rankine-Hugoniot \eqref{RHst} from \eqref{RH2bisphi} in the particular case $\lambda=0$.

The jump conditions compatible with the viscous model \eqref{sys_visquous}, prescribed by the physics, are different from the standard Rankine-Hugoniot conditions \eqref{RHst}. This is due to the presence in \eqref{sys_visquous} of the non-standard viscous term
$$\frac{2}{\rho} \partial_x \theta \partial_x \rho.
$$
In fact, the reader will be easily convinced that, by removing the above viscosity product, the jump conditions compatible with the modified viscous problem turns out to be
\begin{equation}\label{RHmod}
\left\{
\begin{array}{l}
\sigma(\rho_R - \rho_L) = \rho_R \cos(\theta_R) - \rho_L \cos(\theta_L); \\
\smallskip
\displaystyle\sigma(\theta_R - \theta_L) = c\sin(\theta_R) - c\sin(\theta_L) - \lambda \int_{\theta_L}^{\theta_R} \frac{\sin(\theta)^2}{(\cos(\theta) - \sigma)} \, d\theta,
\end{array}
\right.
\end{equation}
to recover the standard Rankine-Hugoniot condition when $\lambda = 0$. Moreover, in this strongly particular case it is possible to find families of paths such that \eqref{RH2bisphi} and \eqref{RHmod} are equivalent.

Moreover, in \eqref{RH2bisG}, the term
$$
\mathcal{G}_\sigma(\theta_R) - \mathcal{G}_\sigma(\theta_L),
$$
can be understood as the contribution of non-conservative products in the vanishing viscosity regime through a shock compatible with the viscous model \eqref{sys_visquous}.

\begin{figure}[htpb]
\begin{center}
			\includegraphics[width=0.8\linewidth]{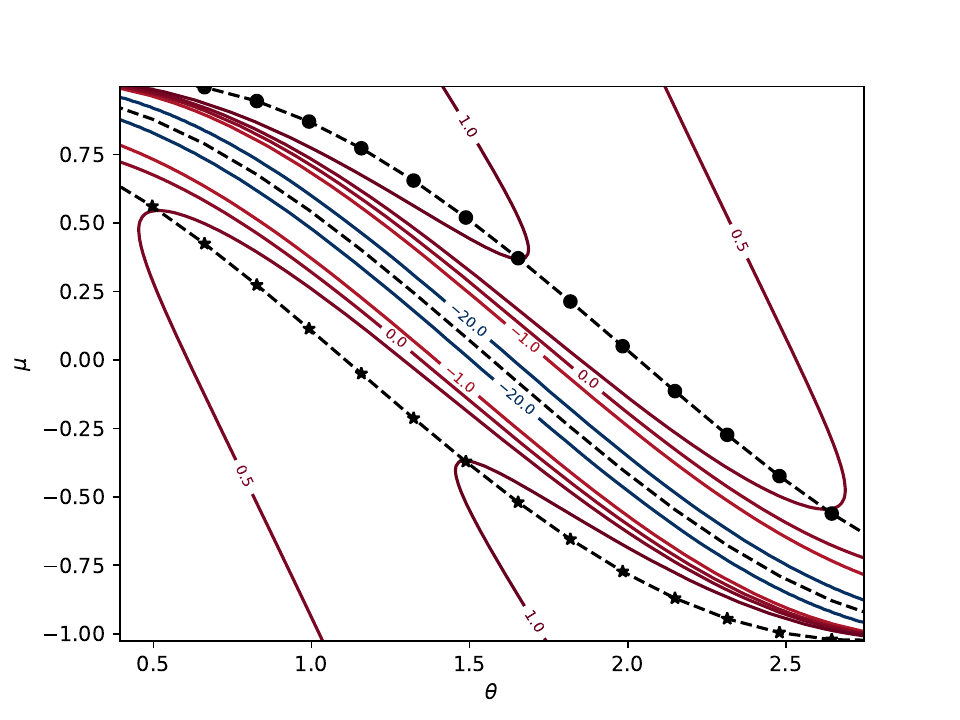}
			\caption{Some contour lines of the function $G_\sigma(\theta)$ for $\theta \in [\theta^*,\pi-\theta^*]$, $\sigma \in [\overline{\mu}_-,1)$  with $c = 0.8$, $\lambda = 0.2$. The curve $\sigma = \mu_-(\theta)$ (dashed black line with stars), $\sigma = cos(\theta)$ (dashed black line) and $\sigma = \mu_+(\theta)$ (dashed black line with circles)  are also shown.}
		\label{fig:graphG}
        \end{center}
	\end{figure}

\subsection{Shock waves}
Equipped with the Generalized Rankine-Hugoniot relations \eqref{RH} and the conditions (SW1)-(SW2)-(SW3), we now exhibit the shock waves and the associated Hugoniot curves. To address such an objective, two given states $U_L, U_R \in \Omega$ are connected by an admissible shock wave as soon as we are able to characterize a shock velocity $\sigma$ such that 
\begin{equation}\label{eqG}
G_\sigma(\theta_R) - G_\sigma(\theta_L) = 0,
\end{equation}
so that (SW1) is satisfied where $G_\sigma$ is defined by \eqref{G_sigma}. As a consequence, $\sigma$ has to be a root of the function
\begin{equation}\label{F(sigma)}
\sigma \in [\overline{\mu}_-, \overline{\mu}_+ ] \to F(\sigma) = G_\sigma(\theta_R) - G_\sigma(\theta_L).
\end{equation}
First, we assume $ \theta_L < \theta_R$.
Before discussing the roots of this function, it is helpful to have an idea of the behavior of  $G_\sigma(\theta)$ as a function of both $\sigma$ and $\theta$. In Figures \ref{fig:graphG} and \ref{fig:graphG2}, we present some contour lines of $G$  in the
regions $(\theta, \sigma) \in [\theta^*, \pi-\theta^*] \times [\overline{\mu}_-, 1)$ and $(\theta, \sigma) \in [\theta^*, \pi-\theta^*] \times (1,\overline{\mu}_+]$ respectively. The curves $\sigma = \mu_\pm(\theta)$, $\sigma = \cos(\theta)$ are also displayed.
We note that the function $G_\sigma$ has two asymptotes, one on the curve $\sigma = \cos(\theta)$ where it goes to $-\infty$ and another one in 
$\sigma = 1$. Moreover, the function $G_\sigma$ diverges to $-\infty$ as $\sigma$ goes to $1$ from above and to $\infty$ as $\sigma$ goes to $1$ from below.

\begin{figure}[htpb]
\begin{center}
			\includegraphics[width=0.8\linewidth]{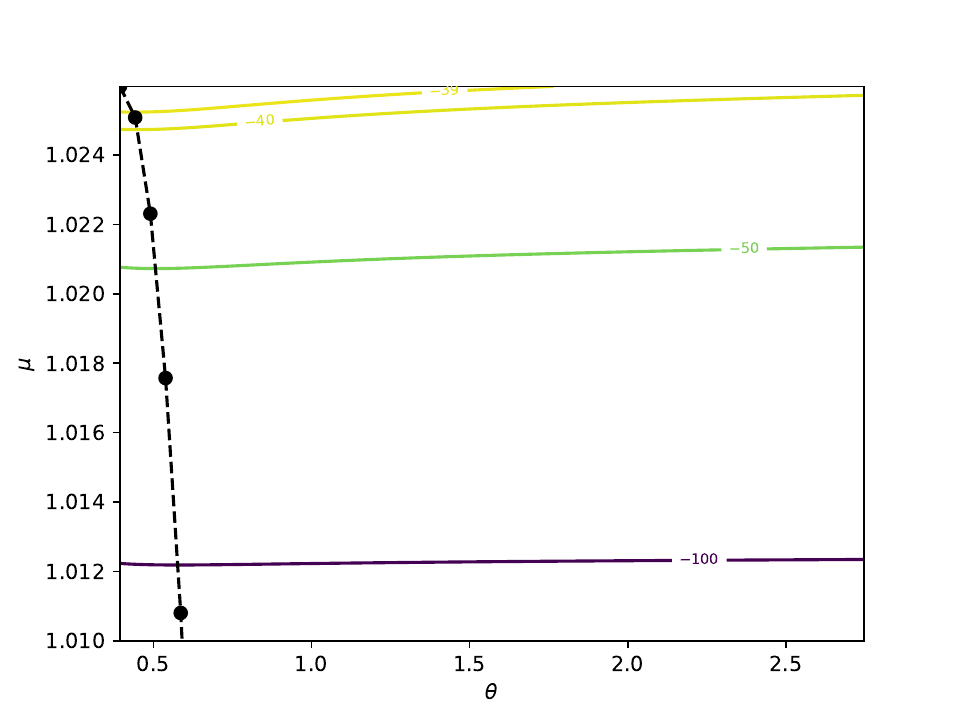}
			\caption{Some contour lines of the function $G_\sigma(\theta)$ for $\theta \in [\theta^*,\pi-\theta^*]$, $\sigma \in (1,\overline{\mu}_+]$ with $c = 0.8$, $\lambda = 0.2$. The curve $\sigma = \mu_+(\theta)$ (dashed black line with circles) is also shown.}
		\label{fig:graphG2}
        \end{center}
\end{figure}

Next, Figure \ref{fig:graphF} displays an example of the graph of a function $F$ given by \eqref{F(sigma)} according to fixed values for $\theta_L$, $\theta_R$ and given parameters $c$ and $\lambda$. The two asymptotes at $\sigma = \cos(\theta_R)$ and $\sigma = \cos(\theta_L)$ are clearly characterized, when it diverges to $-\infty$ and $\infty$ respectively. In addition, we observe that the asymptote in $\sigma = 1$ disappears since we have
\begin{align*}
F(1) &= G_1(\theta_R) - G_1(\theta_L),\\
&=\int_{\theta_L}^{\theta_R} \frac{- 1 + c \cos(\theta) - \frac{\lambda \sin(\theta)^2}{\cos(\theta) - 1}}{(\cos(\theta) - 1)^2} < \infty.
\end{align*}
Observe that $F$ has three real roots. To analyze  which of the roots might correspond to the speed of an admissible shock wave, let us examine the behavior of the function $G_\sigma(\theta)$ depending on the value of $\sigma$.

\begin{figure}[htpb]
\begin{center}
			\includegraphics[width=0.8\linewidth]{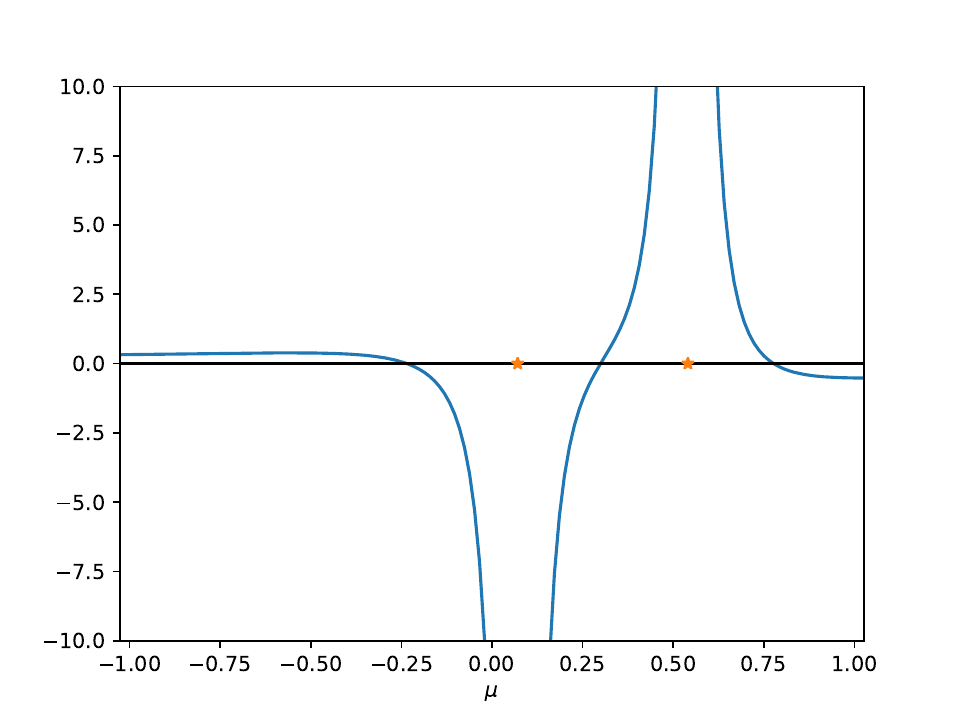}
			\caption{Graph of the function $F(\sigma)$ for $\theta_L = 1$, $\theta_R = 1.5$, $c$ = 0.8,
$\lambda$ = 0.2. $\cos(\theta_L)$ and $\cos(\theta_R)$ are marked in the horizontal axis.}
		\label{fig:graphF}
        \end{center}
\end{figure}

From the information given by Figures \ref{fig:graphG} and \ref{fig:graphG2}, since we have 

\begin{eqnarray*}
G_\sigma'(\theta) = 0 & \iff & (c + \lambda) \cos^2(\theta) - (1 + c)\sigma \cos(\theta) + \sigma^2 - \lambda \\
&\iff& \theta = \chi_\pm(\sigma)
\end{eqnarray*}
 where the functions $\chi_\pm$, given by \eqref{invav}, stand for the inverse of the eigenvalues $\mu_\pm$, we immediately deduce that the behavior of the function
$\theta \to G_\sigma(\theta)$ with 
$\sigma \in [\overline{\mu}_-,\overline{\mu}_+]$ is as follows:
\begin{itemize}
\item If $\sigma \in [-1,1)$:
\begin{itemize}
\item the function $G_\sigma(\theta)$ is increasing in $[\theta^*, \chi_+(\sigma)]$;
\item the function $G_\sigma(\theta)$ is decreasing in $[\chi_+(\sigma), \arccos(\sigma))$;
\item the function $G_\sigma(\theta)$ is increasing in $(\arccos(\sigma), \chi_-(\sigma)]$;
\item the function $G_\sigma(\theta)$ is decreasing in $[\chi_-(\sigma), \pi - \theta^*]$;
\item and
$$\lim_{\theta \to \arccos(\sigma)^\pm} G_\sigma(\theta) = -\infty.$$ 
\end{itemize}
\item If $\sigma < -1$:
\begin{itemize}
\item the function $G_\sigma(\theta)$ is increasing in $[\theta^*, \chi_+(\sigma)]$;
\item the function $G_\sigma(\theta)$ is decreasing in $[\chi_+(\sigma),  \pi - \theta^*]$.
\end{itemize}
\item If $\sigma > 1$:
\begin{itemize}
\item the function $G_\sigma(\theta)$ is increasing in $[\theta^*, \chi_-(\sigma)]$;
\item the function $G_\sigma(\theta)$ is decreasing in $[\chi_-(\sigma),  \pi - \theta^*]$.
\end{itemize}
\end{itemize}
The three above behaviors of the function $G_\sigma$ are presented in Figure \ref{fig:graphG(theta)}. 

\begin{figure}[htpb]
\begin{center}
			\includegraphics[width=0.8\linewidth]{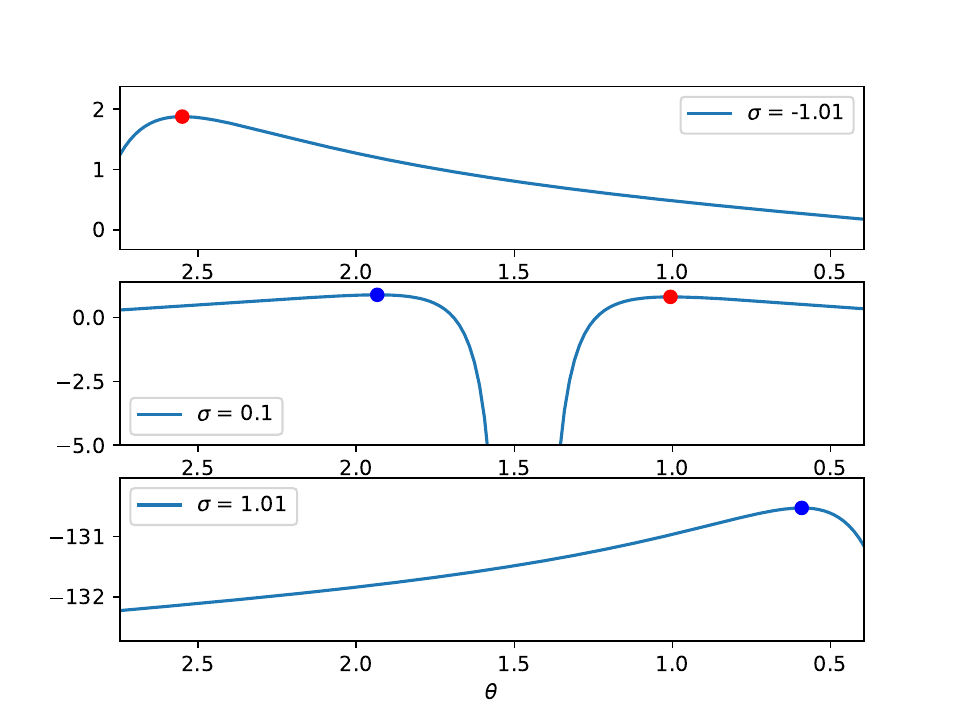}
			\caption{Graph of the functions $G_\sigma(\theta)$ for $\sigma = -1.01$, $\sigma = 0.1$, $\sigma = 1.01$ for $c$ = 0.8,
$\lambda$ = 0.2. The values at the points
$\chi_+(\sigma)$ and $\chi_-(\sigma)$ are marked with red and blue circles respectively.}
		\label{fig:graphG(theta)}
        \end{center}
	\end{figure}

By combining all of the previously properties of $G_\sigma$, we are now able to identify the set of required roots for the function F defined by \eqref{F(sigma)}:

\begin{theorem}
Given $\theta_L, \theta_R \in [\theta^*, \pi- \theta^*]$ with $\theta_L < \theta_R$, the equation \eqref{eqG} has three distinct real roots ordered as follows:
$$\sigma_- < \sigma_0 < \sigma_+.$$
Only $\sigma_-$ and $\sigma_+$ can be the speeds of shock waves compatible with the viscous jump relation (SW1)-(SW2)-(SW3) that link the pair of states $(U_L, U^-)$ and $(U_L, U^+)$ respectively, with
\begin{equation} \label{shocks}
U_L = \left[ \begin{array}{c} \rho_L \\ \theta_L \end{array} \right], 
\quad U^\pm_R = \left[ \begin{array}{c}\displaystyle \rho_L \left(\frac{\sigma_\pm -\cos(\theta_L)}{\sigma_\pm -\cos(\theta_R)} \right) \\ \theta_R \end{array} \right]
\end{equation}
for all given $\rho_L > 0$. Moreover, the shock waves are also compatible with Lax's entropy conditions:
\begin{equation}\label{Lax}
\mu_{\pm}(\theta_R) < \sigma_\pm < \mu_\pm(\theta_L).
\end{equation}
\end{theorem}
\begin{proof}
To begin with, there is always a root between the two asymptotes (see Figure \ref{fig:graphF})
$$\sigma_0 \in (\cos(\theta_R), \cos(\theta_L)).$$
Nevertheless, this root cannot correspond to an admissible shock wave since (SW3) is not satisfied.

On the other hand, since
$ \theta \to G_{\mu_-(\theta_R)}(\theta)$
is an increasing function over $[\theta^*, \theta_R]$, we get
$$ F(\mu_-(\theta_R)) = G_{\mu_-(\theta_R)}(\theta_R) - G_{\mu_-(\theta_R)}(\theta_L)  > 0.$$
If $\mu_-(\theta_L) < \cos(\theta_R)$, knowing that 
$ \theta \to G_{\mu_-(\theta_R)}(\theta)$
is a decreasing function over $[\theta_L, \allowbreak \arccos(\mu_-(\theta_L))]$, we now obtain
$$ F(\mu_-(\theta_L)) = G_{\mu_-(\theta_L)}(\theta_R) - G_{\mu_-(\theta_L)}(\theta_L)  < 0,$$
and necessarily there exists a root within the interval $(\mu_-(\theta_R), \mu_-(\theta_L))$. Otherwise, if $\mu_-(\theta_L) > \cos(\theta_R)$, there is a root in
$[\mu_-(\theta_R), \cos(\theta_R)]$. As a consequence, there is a second root
$$
\sigma_- \in (\mu_-(\theta_R), \min(\cos(\theta_R), \mu_-(\theta_L))).$$
(SW1) is then satisfied for $\sigma_-$. Moreover, let us underline that this root $\sigma_-$ directly satisfies Lax's entropy condition \eqref{Lax}. 

Next, we establish that (SW2) holds true. Since $\chi_-$ is decreasing function with respect to $\sigma$, we immediately get
$$ \theta_L < \chi_-(\sigma_-) < \theta_R.$$
On the other hand,  $\sigma_- < \cos(\theta_R)$ leads to
$$ \theta_R < \arccos(\sigma_-).$$
We have then
$$\theta^* \leq \theta_L < \chi_-(\sigma_-) < \theta_R < \arccos(\sigma^-).$$
Next, since the function
$\theta \to G_{\sigma_-}(\theta)$
is increasing over $[\theta^*, \chi_-(\sigma_-) ]$ and decreasing over $[\chi_-(\sigma_-) , \arccos(\sigma_-) ]$, and since we have
$$G_{\sigma_-}(\theta_L) = G_{\sigma_-}(\theta_R), $$
necessarily we obtain
$$ G_{\sigma_-}(\theta) - G_{\sigma_-}(\theta_L) > 0, \quad \forall \theta \in [\theta_L, \theta_R],
$$
so that (SW2) is satisfied. Therefore, given $\rho_L > 0$, $\sigma_-$ is the speed of an admissible shock wave linking
$U_L = [\rho_L, \theta_L]^T$, $U^-_R = [\rho_R, \theta_L]^T$ if 
$$
\sigma_- (\rho_R - \rho_L) = (\cos(\theta_R) -  \cos(\theta_L)),$$
i.e. if  \eqref{shocks} is satisfied. To conclude with the root $\sigma_-$, let us note that, by construction of the function $F$ given by \eqref{F(sigma)}, the condition (SW3) is also satisfied.

A similar reasoning shows that there is a third root
$$
\sigma_+ \in [\max(\mu_+(\theta_R),\cos(\theta_L)), \mu_+(\theta_L)],$$
satisfying both (SW1)-(SW2)-(SW3) and Lax's entropy condition \eqref{Lax} connecting pair of states of the form \eqref{shocks}. The proof is thus completed.
\end{proof}

Next, assuming $\theta_L > \theta_R$, a similar reasoning can be followed by simply exchanging the roles of $\theta_L$ and $\theta_R$. As a consequence, once again three roots can be exhibited but now $\sigma_\pm$ satisfy 
$$
\mu_-(\theta_L) < \sigma_\pm < \mu_-(\theta_R)$$
and
$$
G_{\sigma_\pm} (\theta) - G_{\sigma_\pm}(\theta_L) \geq 0, \quad \forall \theta \in (\theta_R, \theta_L)$$
so that the corresponding shock waves are not admissible. 

Therefore, given a state $U_L \in \Omega$, the Hugoniot curve composed by the states that can be connected by an admissible shocks are given by 
\begin{eqnarray}
\nonumber \mathcal{S}_-(U_L) & = & \Biggl\{\left[ \begin{array}{c}\displaystyle \rho_L \left(\frac{\sigma_- -\cos(\theta_L)}{\sigma_- -\cos(\theta_R)} \right) \\ \theta_R \end{array} \right], \quad \theta_R \in (\theta_L, \pi-\theta^*) \\
& &\qquad  \qquad \qquad \text{ where $\sigma_-$ is the smallest root of \eqref{eqG}} \Biggr\}, \label{S-}\\
\nonumber \mathcal{S}_+(U_L) & = & \Biggl\{\left[ \begin{array}{c}\displaystyle \rho_L \left(\frac{\sigma_+ -\cos(\theta_L)}{\sigma_+ -\cos(\theta_R)} \right) \\ \theta_R \end{array} \right], \quad \theta_R \in (\theta_L, \pi-\theta^*) \\
& & \qquad \qquad \qquad \text{ where $\sigma_+$ is the largest root of \eqref{eqG}} \Biggr\}, \label{S+}
\end{eqnarray}
Examples of the curves $\mathcal{S}_\pm(U_L)$ are displayed in Figure \ref{fig:Simple_curves}.

\begin{figure}[htpb]
\begin{center}
			\includegraphics[width=0.7\linewidth]{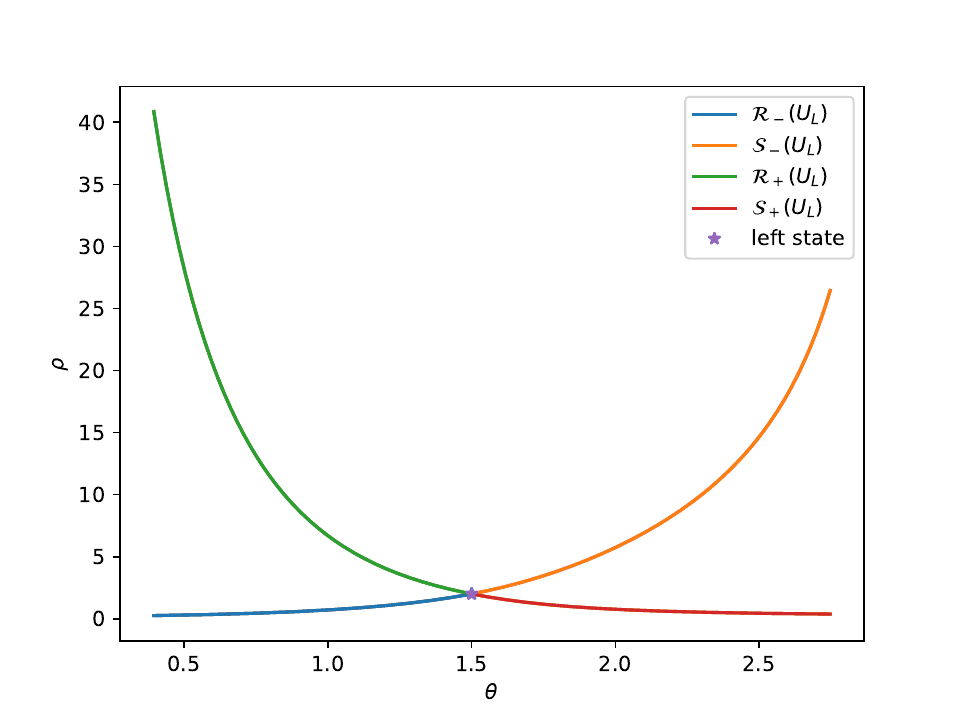}
			\caption{Simple  curves for the left state $U_L = [2, 1.5]^T$ for $c$ = 0.8,
$\lambda$ = 0.2. }
		\label{fig:Simple_curves}
        \end{center}
	\end{figure}

As it has been said in the Introduction, the Rankine-Hugoniot condition \eqref{RH_con} corresponding to the conservative formulation of the system \eqref{scl} is not equivalent to \eqref{RH}: in effect, Figure \ref{fig:Hugoniot_comparison} shows a comparison of  $S_-(U_L)$ with the curve $S_-^{\text{con}}(U_L)$ formed by the right states that can be connected with $U_L$ by admissible shocks satisfying \eqref{RH_con}, which is given by:
\begin{eqnarray}\nonumber 
\mathcal{S}_-^{\text{con}}(U_L) & = & \Biggl\{\left[ \begin{array}{c}\displaystyle \rho_L \left(\frac{\sigma_-^{\text{con}} -\cos(\theta_L)}{\sigma_-^{\text{con}} -\cos(\theta_R)} \right) \\ \theta_R \end{array} \right], \quad \theta_R \in (\theta_L, \pi-\theta^*) 
\\& &\qquad \text{where $\sigma_-^{\text{con}}$ is the smallest root of} \nonumber
\\& &\qquad c \big(f_2(\theta_R) - f_2(\theta_L)\big) - \lambda \ln\left(\frac{\cos(\theta_L) - \sigma}{\cos(\theta_R) - \sigma}\right) - \sigma \big(f_1(\theta_R) - f_1(\theta_L)\big) = 0 \Biggr\}.\nonumber
\end{eqnarray}
The curves in the Figure \ref{fig:Hugoniot_comparison} corresponds to $U_L = [1,1]^T$ and it can be seen that they are quite different for shocks of large amplitude.

\begin{figure}[htpb]
\begin{center}
			\includegraphics[width=0.7\linewidth]{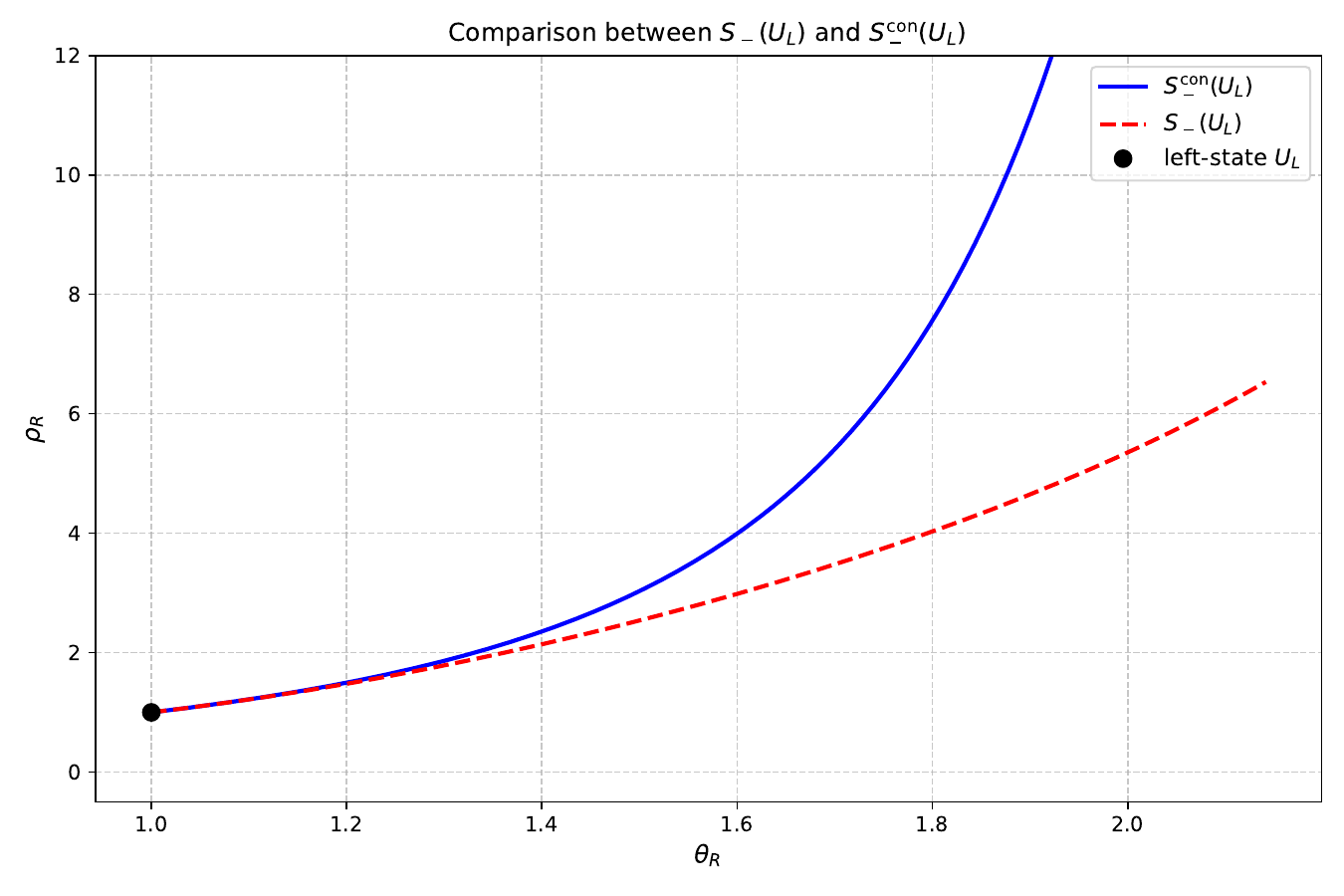}
			\caption{Comparison between $S_-(U_L)$ and $S_-^{\text{con}}(U_L)$ for the left state $U_L = [1, 1]^T$ with $c$ = 0.8,
$\lambda$ = 0.2. }
		\label{fig:Hugoniot_comparison}
        \end{center}
	\end{figure}

%%%%%%%%%%%%%%%%%%%%%%%%%%%%%%%%%%%%%%%%%%%%%%%%%%%%%%%%%%%
%%%%%%%%%%%%%%%%%%%%%%%%%%%%%%%%%%%%%%%%%%%%%%%%%%%%%%%%%%%
%%%%%%%%%%%%%%%%%%%%%%%%%%%%%%%%%%%%%%%%%%%%%%%%%%%%%%%%%%%

\section{The Riemann problem}
In this section we discuss the solution of the Riemann problem associated to \eqref{pde} with initial conditions
$$
U(x,0 )  = \begin{cases}
    U_L & \text{if $x < 0$;} \\
    U_R & \text{ if $x > 0$.}
\end{cases}$$
To address this problem, we consider the  curve
$$
\mathcal{C}_-(U_L ) = \mathcal{R}_-(U_L) \cup \mathcal{S}_-(U_L).$$
formed by the possible right-states of simple waves (rarefactions or shock waves) that have $U_L$ as left state. We also introduce the curve of the possible left states of simple curves that have $U_R$ as right state as follows:
$$
\mathcal{C}_+^b(U_R) = \mathcal{R}^b_+(U_R) \cup \mathcal{S}^b_+(U_R)$$
where
\begin{equation}\label{Rb+}
\mathcal{R}^b_+(U_R) = \left\{ \left[\begin{array}{c} \rho_R\exp(F_+(\theta_L)-F_+(\theta_R)) \\ \theta_L \end{array} \right], \quad \theta_L \in [\theta_R, \pi - \theta^*]\right\},
\end{equation}
\begin{eqnarray}
\nonumber\mathcal{S}^b_+(U_R) & = & \Biggl\{\left[ \begin{array}{c}\displaystyle \rho_R \left(\frac{\sigma_+ -\cos(\theta_R)}{\sigma_+ -\cos(\theta_L)} \right) \\ \theta_L \end{array} \right], \quad \theta_L \in (\theta^*, \theta_R) \\
& & \qquad \qquad \qquad \text{ where $\sigma_+$ is the largest root of \eqref{eqG}} \Biggr\}. \label{Sb+}
\end{eqnarray}

In order to solve the Riemann problem, we need to find the intersection of the curves $\mathcal{C}_-(U_L )$ and $\mathcal{C}_+^b(U_R)$. In other words, we have now to consider the intersections of the graphs of the two following functions:
$$
\mathcal{F}_-(\theta) = \begin{cases}
\rho_L\exp(F_-(\theta)-F_-(\theta_L)) & \text{if $\theta \leq \theta_L$;}\\
\displaystyle \rho_L \left(\frac{\sigma_- -\cos(\theta_L)}{\sigma_- -\cos(\theta_R)} \right) & \text{otherwise;}
\end{cases}
$$
and 
$$
\mathcal{F}_+^b(\theta) = \begin{cases}
\displaystyle\rho_R \left(\frac{\sigma_+ -\cos(\theta_R)}{\sigma_+ -\cos(\theta)} \right) & \text{if $\theta \leq \theta_R$;}\\
\rho_R\exp(F_+(\theta)-F_+(\theta_R)) & \text{otherwise.}
\end{cases}
$$
Since $\mathcal{F}_-$ is an increasing function and $\mathcal{F}^b_+$ is a decreasing one, these two functions may eventually not intersect. As a consequence, the Riemann problem may have no solution in $\Omega$ although it can have one or more solutions in $(0, \infty) \times (-\pi, \pi]$. This non-existence of solution may occur if $$\mathcal{F}_-(\pi-\theta^*) < \mathcal{F}^b_+(\pi-\theta^*) \quad \text{or} \quad\mathcal{F}^b_+(\theta^*) < \mathcal{F}_-(\theta^*).$$ This situation is exemplified in Figure \ref{fig:PR_case1}.
\begin{figure}[htpb]
  \centering
			\includegraphics[width=0.45\linewidth]{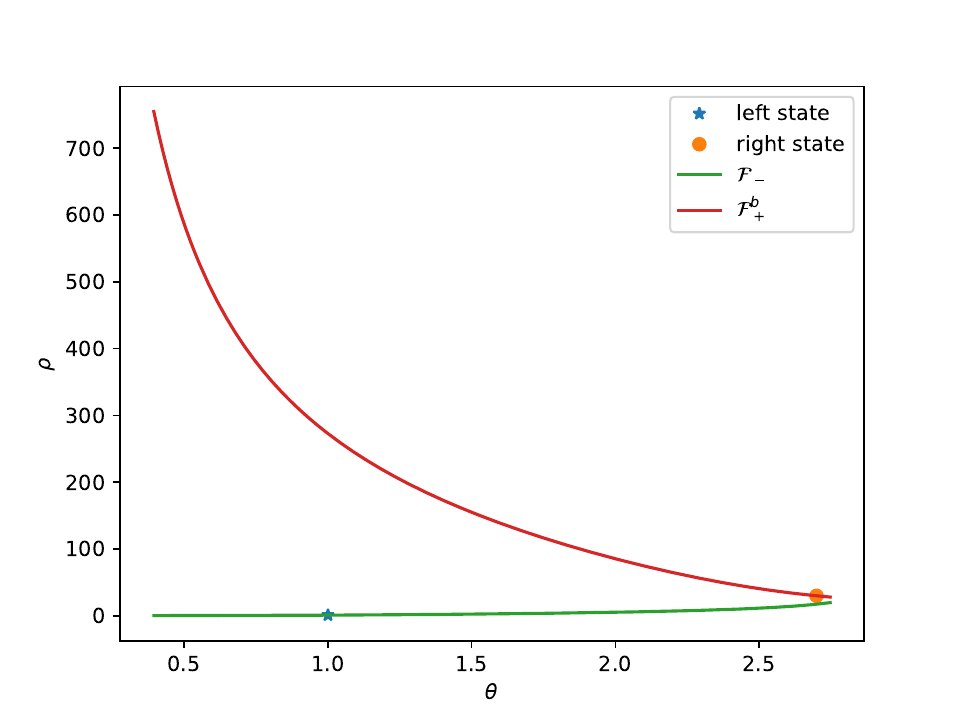}
			\includegraphics[width=0.45\linewidth]{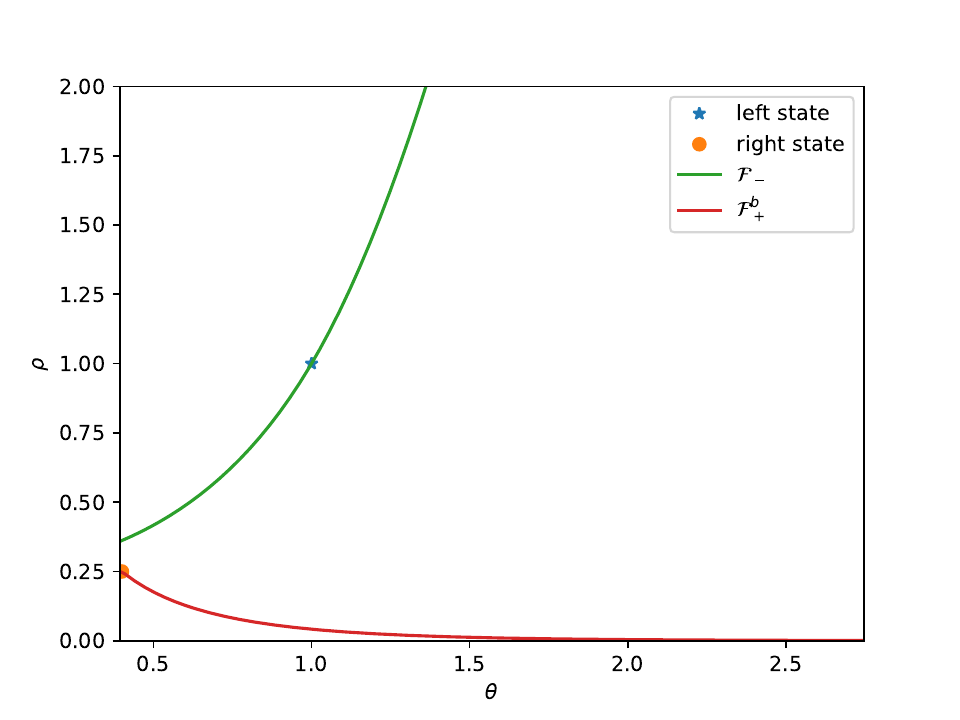}
	\caption{Riemann problem: examples of no solution in $\Omega$. Curves $\mathcal{C}_-(U_L)$ and $\mathcal{C}^b_+(U_R)$ for $U_L = [1,1]^T$, $U_R = [30,2.7]^T$ (left) and $U_L = [1,1]^T$, $U_R = [0.25,0.4]^T$ (right) and parameters $c$ = 0.8, $\lambda$ = 0.2.}
	\label{fig:PR_case1}
\end{figure}

Next, in all other cases, there is a unique $\theta_m \in [\theta^* , \pi - \theta^*]$ such that
$$\mathcal{F}_-(\theta_m) = \mathcal{F}_+^b(\theta_m).$$
Then the structure of the classical solution depends on the relative position of the intermediate state $\theta_m$ with respect to 
$\theta_L$ and $\theta_R$. The four possible configurations are now detailed and illustrated with the parameters given by $c = 0.8$ and $\lambda = 0.2$.

\begin{enumerate}

\item If $\theta_L < \theta_m < \theta_R$ the solution of the Riemann problem consists of two shocks. 
\begin{itemize}
\item \textbf{Example:} $U_L =[1,1]^T$, $U_R = [1,2]^T$. In this case, the intermediate state is 
$U_m = [2.580231243693331, 1.5098899053692902]^T$ and the solution is given by 
$$
U(x,t) = \begin{cases} 
U_L & \text{if $x < \sigma_ t$;}\\
U_m & \text{if $\sigma_- t < x < \sigma_+ t$;}\\
U_R & \text{if $x > \sigma_+t$;}
\end{cases}
$$
with 
$$\sigma_- = -0.24252576951298943, \quad \sigma_+ = 0.3627332052302428.$$
In Figure \ref{fig:PR_2shocks} we display the curves $\mathcal{C}_-(U_L)$ and $\mathcal{C}^b_+(U_R)$ and the graph of the solution at time $t = 1$.
\begin{figure}[htpb]
	\begin{subfigure}[h]{0.45\textwidth}
		\centering
			\includegraphics[width=\linewidth]{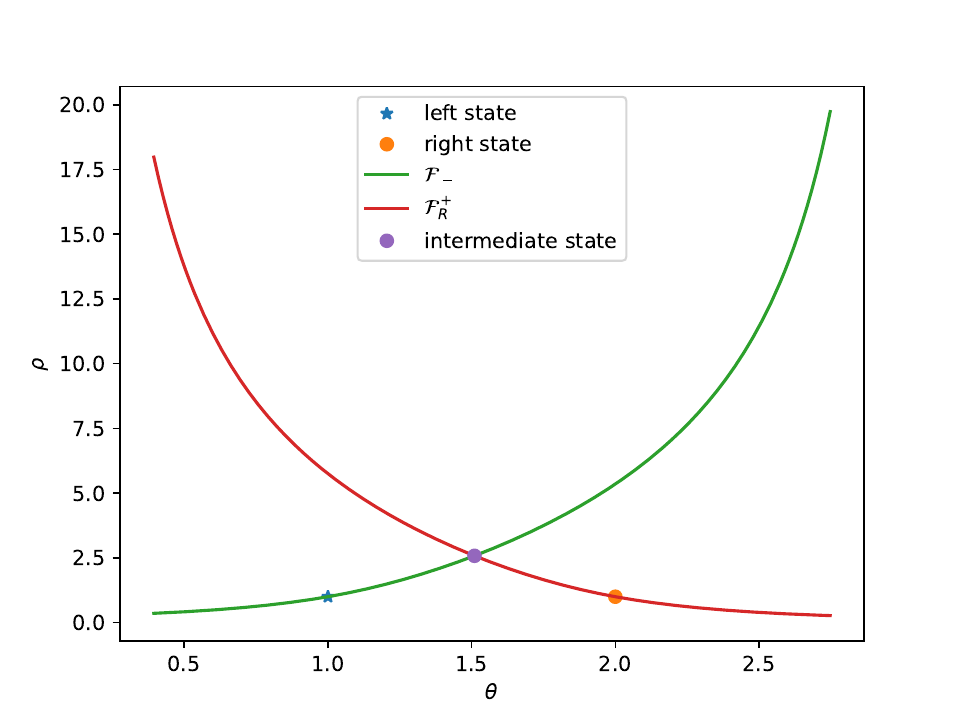}
	\end{subfigure}
	\begin{subfigure}[h]{0.5\textwidth}
		\centering
			\includegraphics[width=1.\linewidth]{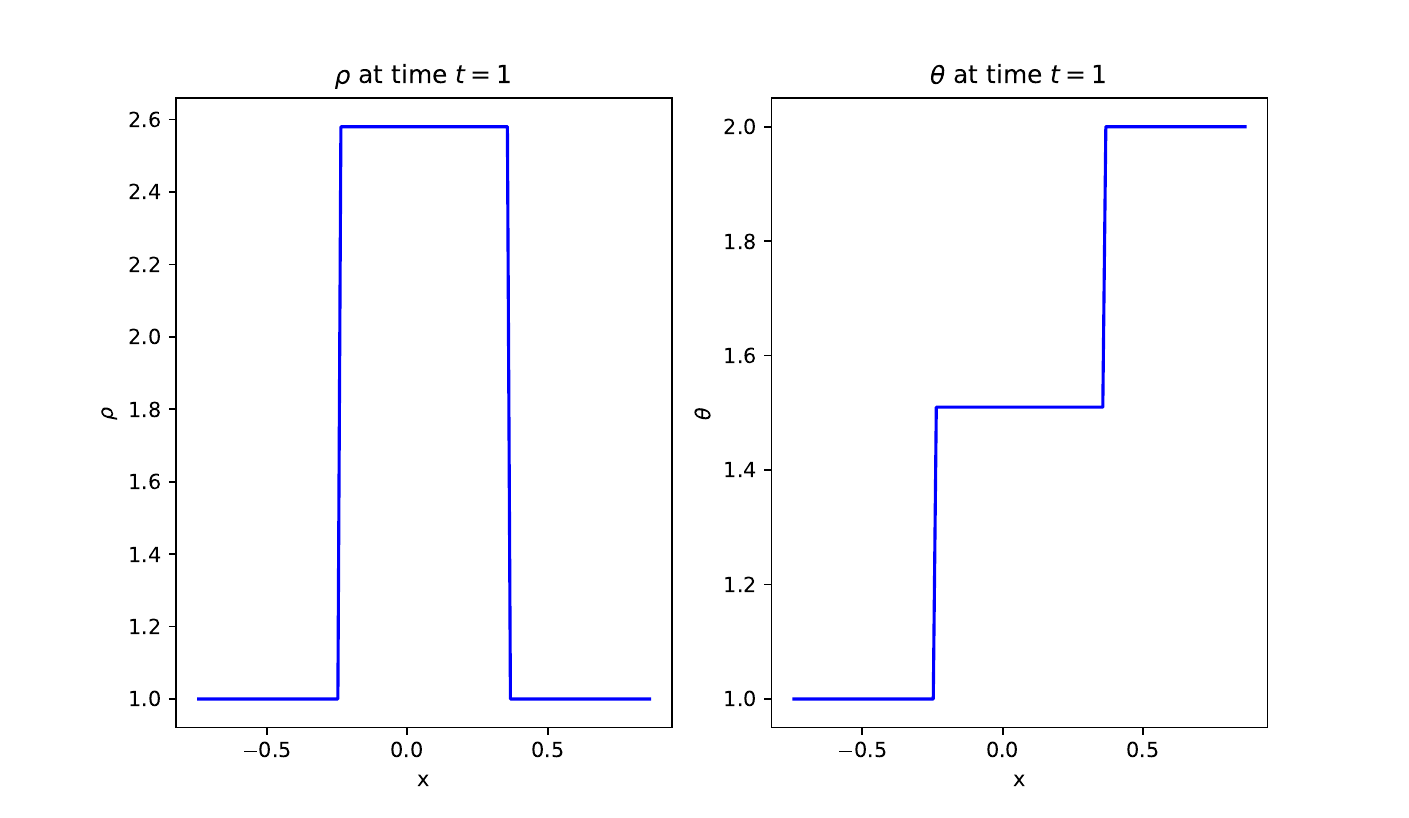}
	\end{subfigure}
	\caption{Riemann problem: two shock waves. Left: curves $\mathcal{C}_-(U_L)$ and $\mathcal{C}^b_+(U_R)$ for $U_L = [1,1]^T$, $U_R = [1, 2]^T$. Right: $\rho$ (left) and $\theta$ (right) at time $t = 1$ and parameters $c$ = 0.8, $\lambda$ = 0.2.}
	\label{fig:PR_2shocks}
\end{figure}

\end{itemize}

\item If $\theta_R < \theta_m < \theta_L$ the solution of the Riemann problem consists of two rarefaction waves. 
\begin{itemize}
\item \textbf{Example:} $U_L =[1,2]^T$, $U_R = [0.5,1]^T$. In this case, the intermediate state is 
$U_m = [0.21862670530690873, 1.335917425836824]^T$ and the solution is given by 
$$
U(x,t) = \begin{cases} 
U_L & \text{if $x \leq \mu_-(\theta_L) t$;}\\
 \left[\begin{array}{c} \rho_L e^{(F_-(\chi_+(\frac{x}{t}))-F_-(\theta_L))} \\ \chi_+(\frac{x}{t}) \end{array} \right] & \text{if $\mu_-(\theta_L) t \leq x \leq \mu_-(\theta_R)t$;}\\
U_m & \text{if $\mu_-(\theta_R)t \leq x \leq \mu_+(\theta_L)t$;}\\
\left[\begin{array}{c} \rho_m e^{(F_+(\chi_-(\frac{x}{t}))-F_+(\theta_m))} \\ \chi_-(\frac{x}{t}) \end{array} \right] & \text{if $\mu_+(\theta_L)t \leq x \leq \mu_+(\theta_R) \sigma_+t$;}\\
U_R & \text{if $\mu_+(\theta_R)t \leq x $.} 
\end{cases}
$$
In Figure \ref{fig:PR_2raref} we display the curves $\mathcal{C}_-(U_L)$ and $\mathcal{C}^b_+(U_R)$ and the graph of the solution at time $t = 1$.
\begin{figure}[htpb]
	\begin{subfigure}[h]{0.45\textwidth}
		\centering
			\includegraphics[width=\linewidth]{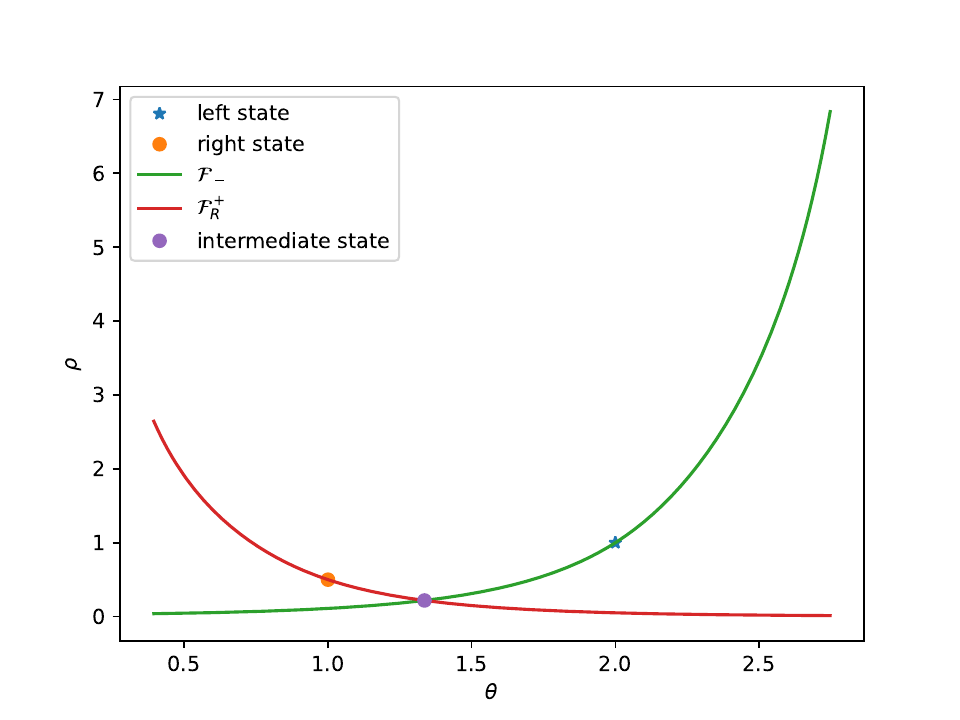}
	\end{subfigure}
	\begin{subfigure}[h]{0.5\textwidth}
		\centering
			\includegraphics[width=1.\linewidth]{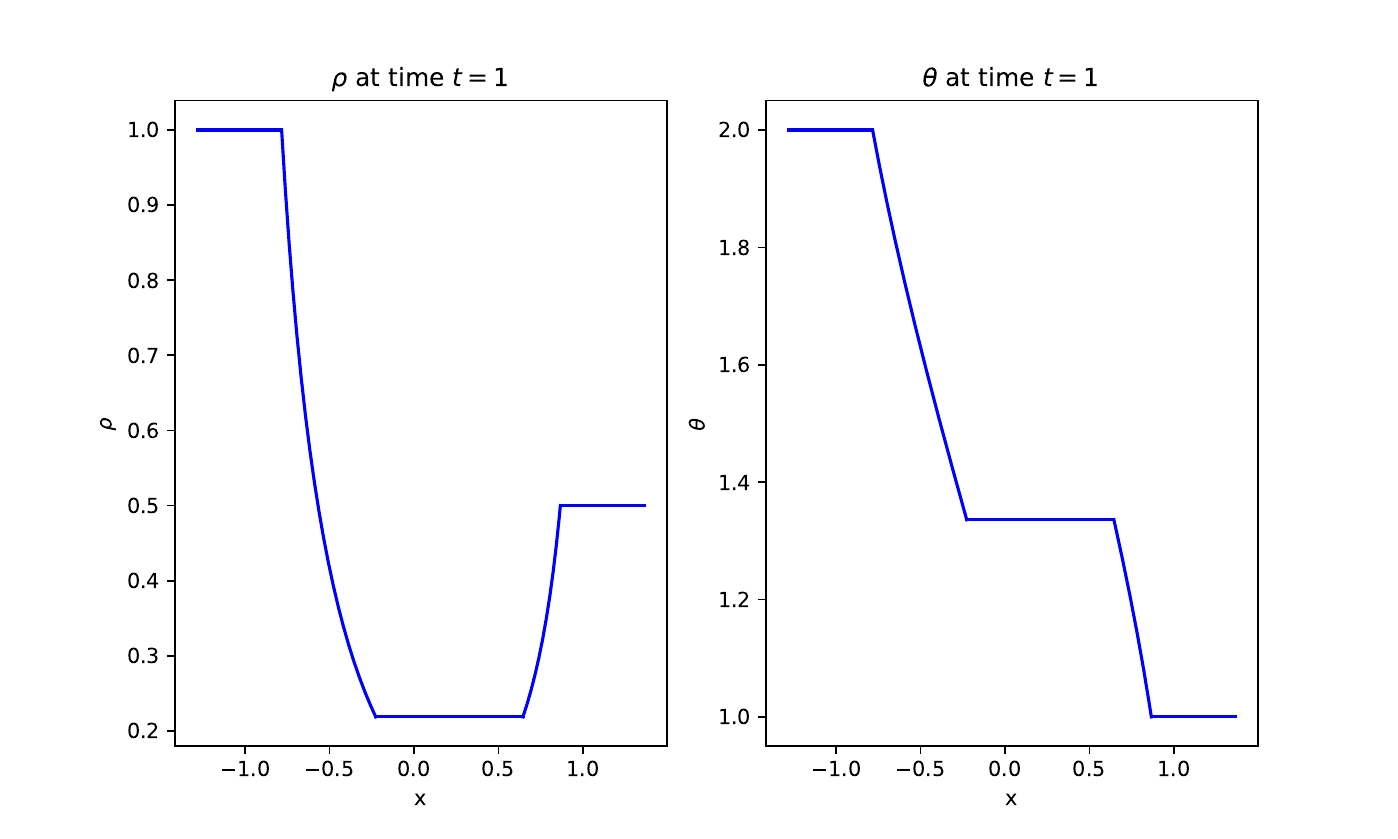}
	\end{subfigure}
	\caption{Riemann problem: two rarefaction waves. Left: curves $\mathcal{C}_-(U_L)$ and $\mathcal{C}^b_+(U_R)$ for $U_L = [1,2]^T$, $U_R = [0.5,1]^T$. Right: $\rho$ (left) and $\theta$ (right) at time $t = 1$ and parameters $c$ = 0.8, $\lambda$ = 0.2.}
	\label{fig:PR_2raref}
\end{figure}

\end{itemize}

\item If $\theta_m < \theta_L$ and $ \theta_m < \theta_R$ the solution of the Riemann problem consists of a 1-rarefaction and a 2-shock wave. 
\begin{itemize}
\item \textbf{Example:} $U_L =[5,2]^T$, $U_R = [0.5,1.5]^T$. In this case, the intermediate state is 
$U_m = [0.8974789258204603 , 1.2418210093791815]^T$ and the solution is given by 
$$
U(x,t) = \begin{cases} 
U_L & \text{if $x \leq \mu_-(\theta_L) t$;}\\
 \left[\begin{array}{c} \rho_Le^{(F_-(\chi_+(\frac{x}{t}))-F_-(\theta_L))} \\ \chi_+(\frac{x}{t}) \end{array} \right] & \text{if $\mu_-(\theta_L) t \leq x \leq \mu_-(\theta_R)t$;}\\
U_m & \text{if $\mu_-(\theta_R)t \leq x \leq \sigma_+ t$;}\\
U_R & \text{if $\sigma_+t \leq x $,} 
\end{cases}
$$
with $\sigma_+ = 0.6404943992157477$.
The curves $\mathcal{C}_-(U_L)$ and $\mathcal{C}^b_+(U_R)$ and the graph of the solution at time $t = 1$ are displayed in Figure \ref{fig:PR_raref+shock}.
\begin{figure}[htpb]
	\begin{subfigure}[h]{0.45\textwidth}
		\centering
			\includegraphics[width=\linewidth]{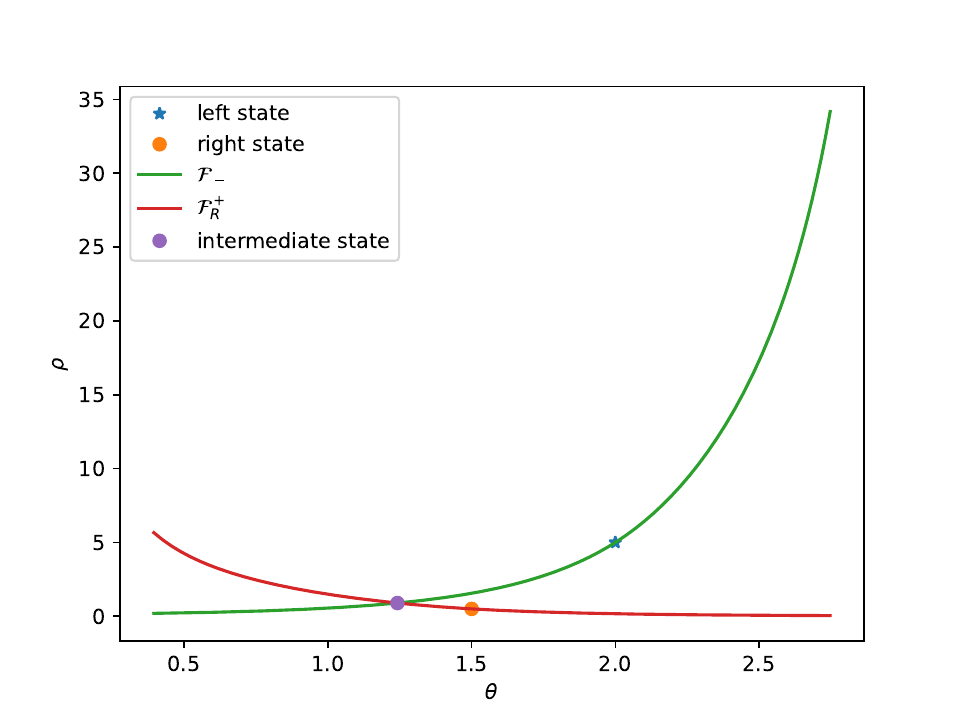}
	\end{subfigure}
	\begin{subfigure}[h]{0.5\textwidth}
		\centering
			\includegraphics[width=1.\linewidth]{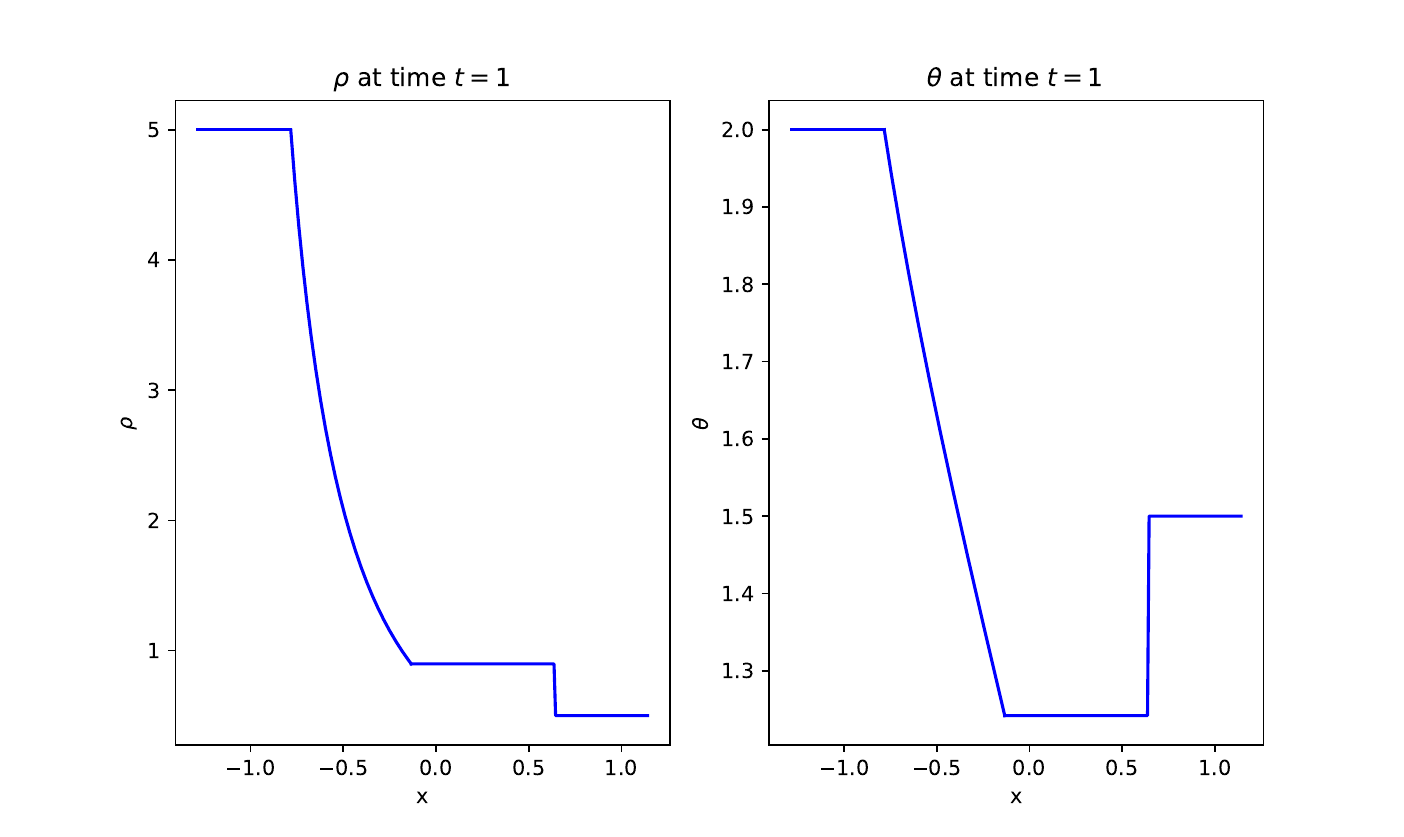}
	\end{subfigure}
	\caption{Riemann problem: rarefaction and shock waves. Left: curves $\mathcal{C}_-(U_L)$ and $\mathcal{C}^b_+(U_R)$ for $U_L = [5,2]^T$, $U_R = [0.5, 1.5]^T$. Right: $\rho$ (left) and $\theta$ (right) at time $t = 1$ and parameters $c$ = 0.8, $\lambda$ = 0.2.}
	\label{fig:PR_raref+shock}
\end{figure}

\end{itemize}

\item If $\theta_m > \theta_L$ and $ \theta_m > \theta_R$ the solution of the Riemann problem consists of a 1-shock and a 2-rarefaction waves. 
\begin{itemize}
\item \textbf{Example:} $U_L = [0.5,1.5]^T$ and $U_R = [5,2]^T$. In this case, the intermediate state is 
$U_m = [2.64567223988045, 2.3314594288198824]^T$ and the solution is given by 
%$$
%U(x,t) = \begin{cases} 
%U_L & \text{if $x \leq \sigma_- t$;}\\
%U_m & \text{if $\sigma_-t \leq x \leq \mu_+(\theta_L)t$;}\\
%\left[\begin{array}{c} \rho_m\exp(F_+(\chi_-(\frac{x}{t}))-F_+(\theta_m)) \\ \chi_-(\frac{x}{t}) \end{array} \right] & \text{if $\mu_+(\theta_L)t \leq x \leq \mu_+(\theta_R) \sigma_+t$;}\\
%U_R & \text{if $\mu_+(\theta_R)t \leq x $.} 
%\end{cases}
%$$
\[
U(x,t) =
\begin{cases}
  U_L & \text{if } x \leq \sigma_- t;\\
  U_m & \text{if } \sigma_- t \leq x \leq \mu_+(\theta_L)\,t;\\[4pt]
  \begin{bmatrix}
    \rho_m\, e^{(F_+(\chi_-(x/t)) - F_+(\theta_m))} \\
    \chi_-(x/t)
  \end{bmatrix}
  & \text{if } \mu_+(\theta_L)\,t \leq x \leq \mu_+(\theta_R)\,\sigma_+ t;\\[4pt]
  U_R & \text{if } x \geq \mu_+(\theta_R)\,t.
\end{cases}
\]
with $\sigma_- = -0.8665350302618553$.
The curves $\mathcal{C}_-(U_L)$ and $\mathcal{C}^b_+(U_R)$ and the graph of the solution at time $t = 1$ are displayed in Figure \ref{fig:PR_shock+raref}.
\begin{figure}[htpb]
	\begin{subfigure}[h]{0.45\textwidth}
		\centering
			\includegraphics[width=\linewidth]{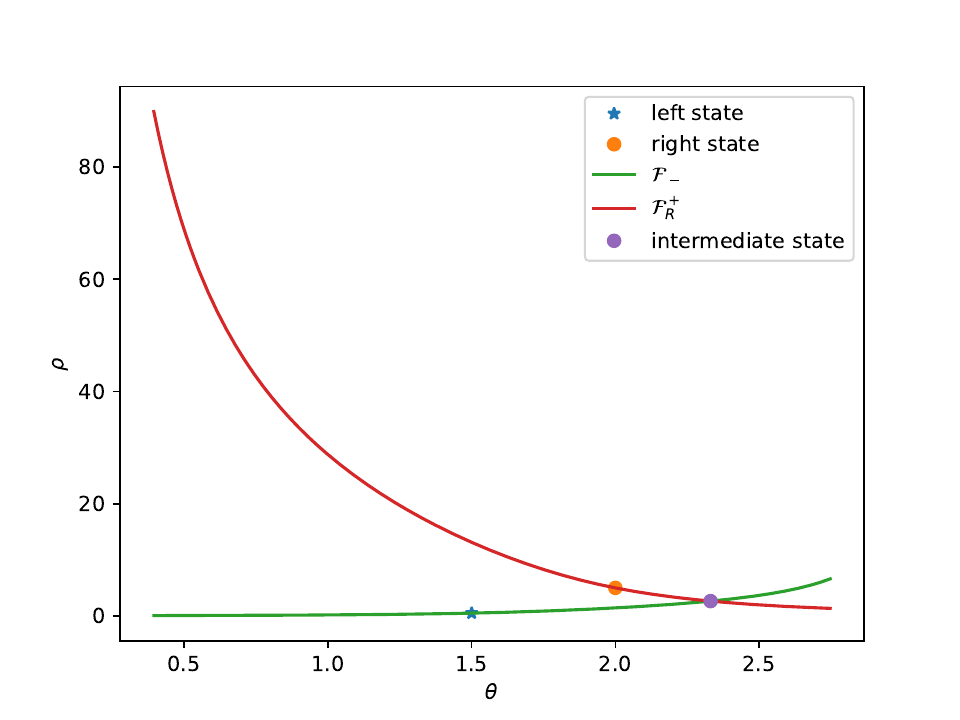}
	\end{subfigure}
	\begin{subfigure}[h]{0.5\textwidth}
		\centering
			\includegraphics[width=1.\linewidth]{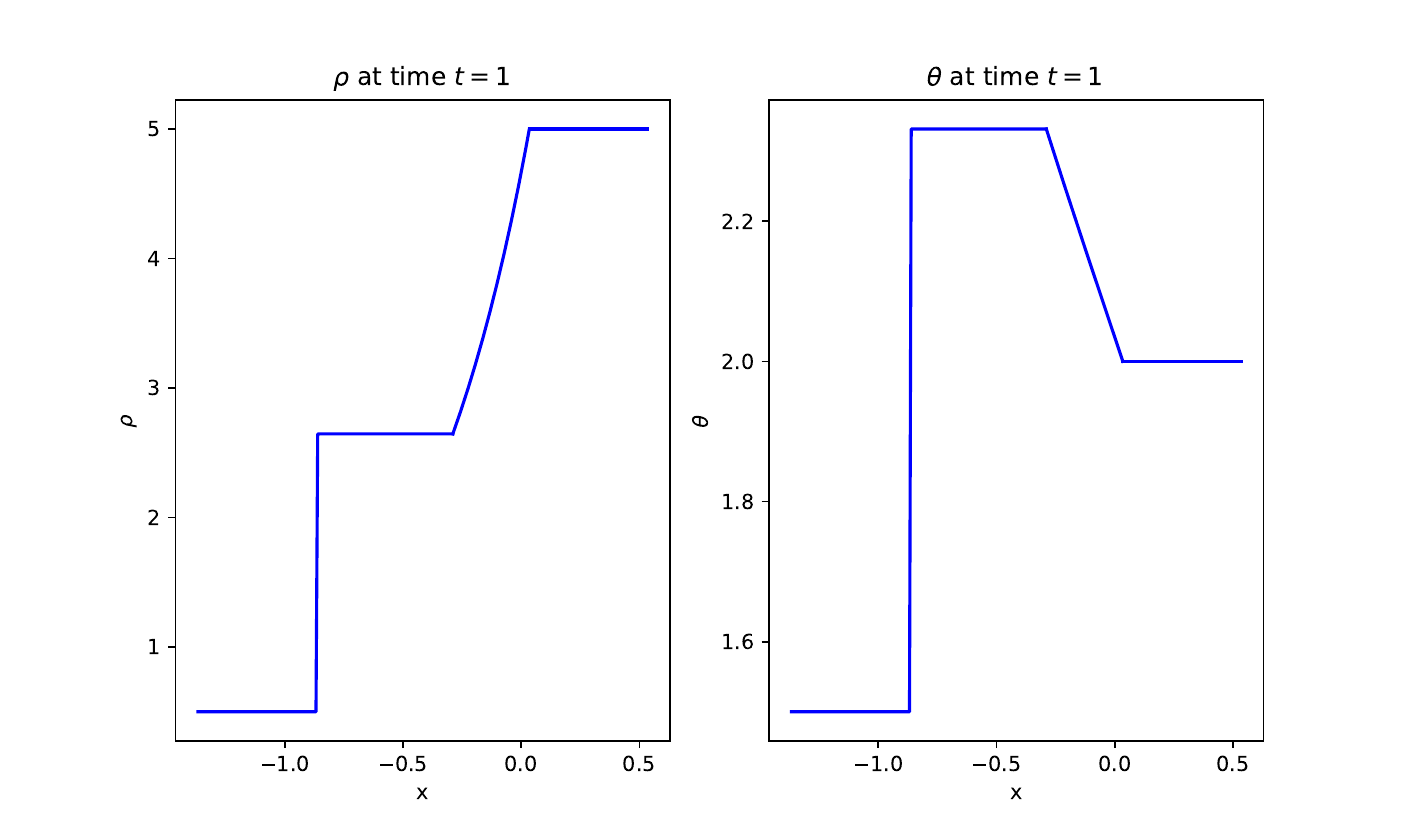}
	\end{subfigure}
	\caption{Riemann problem: rarefaction and shock waves. Left: curves $\mathcal{C}_-(U_L)$ and $\mathcal{C}^b_+(U_R)$ for $U_L = [0.5, 1.5]^T$ and $U_R = [5,2]^T$. Right: $\rho$ (left) and $\theta$ (right) at time $t = 1$ and parameters $c$ = 0.8, $\lambda$ = 0.2.}
	\label{fig:PR_shock+raref}
\end{figure}

\end{itemize}

\end{enumerate}

\section{Conclusions and future perspectives}

In this work, we have studied the solutions of the Riemann problem for the Self-Organized Hydrodynamics model developed in \cite{Degond2008}, specifically in the formulation presented in \cite{motsch2011numerical, compain:hal-05490904}. By incorporating a physically motivated viscosity \cite{Degond2012}, exact jump conditions across shock waves can be established, as shown in \cite{compain:hal-05490904}. We have proven that these conditions cannot be obtained by the prescription of a family of paths using the theory of Dal Maso, LeFloch and Murat \cite{DMLFM1995} due to a non standard viscosity term. These non standard generalized Rankine-Hugoniot conditions, together with the Riemann invariants, enable the full characterization of both shock and rarefaction wave curves. Based on these wave curves, we construct the complete solution to the Riemann problem, focusing on regions where the characteristic fields are genuinely nonlinear. Some examples are shown for the different configurations. Exact solutions to the Riemann problem are available to compute at: \href{https://mariecmpn.github.io/website/riemann_solver_soh.html }{https://mariecmpn.github.io/website/riemann\_solver\_soh.html}
As a future work, the proposed Riemann solver will be used to design a numerical scheme with in-cell discontinuous reconstructions \cite{chalons2020path, pimentelgarcia2022cell, pimentel2024drmood} capable of converging to the physical solution in the presence of shocks.

%%%%%%%%%%%%%%%%%%%%%%%%%%%%%%%%%%%%%%%%%%%%%%%%%%%%%%%%%%%
%%%%%%%%%%%%%%%%%%%%%%%%%%%%%%%%%%%%%%%%%%%%%%%%%%%%%%%%%%%
%%%%%%%%%%%%%%%%%%%%%%%%%%%%%%%%%%%%%%%%%%%%%%%%%%%%%%%%%%%

\appendix

\section{Integrals}

\subsection{Primitive function for the integral curve}\label{appendix_1}
The purpose of this appendix is to derive an explicit expression for the primitive function needed in the computation of the integral curves \eqref{charcurv} for rarefaction wave solutions.
$$
\int \frac{\sin s }{\cos s - \lambda_\pm(s)}\, ds = F_\pm(\theta) + constant,
$$
with        %\begin{equation}
        %F_\pm(x) = -\frac{a}{\lambda} \log \left| \sqrt{a^2\cos^2(x) + \lambda\sin^2(x)} \mp a\cos(x) \right| \pm \frac{\sqrt{\lambda - a^2}}{\lambda} \arcsin\left( \cos(x) \sqrt{\frac{\lambda - a^2}{\lambda}} \right) 
        %\end{equation}
\begin{multline}
  F_\pm(x) = -\frac{a}{\lambda}
    \log \left| \sqrt{a^2\cos^2(x) + \lambda\sin^2(x)} \mp a\cos(x) \right| \\
    \pm \frac{\sqrt{\lambda - a^2}}{\lambda}
    \arcsin\left( \cos(x) \sqrt{\frac{\lambda - a^2}{\lambda}} \right),
\end{multline}
        and
        $$
a^2 = \left(\frac{1-c}{2}\right)^2.
$$
Functions $F_\pm$ were computed using AI.

\subsection{Primitive function for the jump conditions}\label{appendix_2}
Now, in this appendix, we derive an explicit formulation for the primitive function needed in the evaluation of the Rankine-Hugoniot relations \eqref{G_sigma} for shock wave solutions.
\begin{eqnarray*}
G_\sigma(\theta) & = & \int \frac{- \sigma + c \cos(\theta) - \frac{\lambda \sin(\theta)^2}{\cos(\theta) - \sigma}}{(\cos(\theta) - \sigma)^2}\, d\theta\\
& = & - \frac{\lambda \sin(\theta)}{2(\cos(\theta) - \sigma)^2} + 
\frac{\sigma (2c + \lambda - 2) \sin(\theta)}{2 (1 - \sigma^2) (\cos(\theta) - \sigma)}
+  \frac{2c + \lambda - 2\sigma^2}{2 (1 - \sigma^2)} J(\theta, \sigma)
\end{eqnarray*}
with
$$
J(\theta, \sigma) = \begin{cases}
\frac{2}{\sqrt{1 - \sigma^2}}\mathrm{arctanh}\left( \sqrt{\frac{1 + \sigma}{1 - \sigma}}t \right)
  & \text{if $|\sigma| < 1$;}\\
  -\frac{2}{\sqrt{\sigma^2 -1}}\arctan\left( \sqrt{\frac{\sigma + 1}{\sigma-1}}t\right)
  & \text{if $\sigma > 1$;}\\
    \frac{2}{\sqrt{\sigma^2 -1}}\arctan\left( \sqrt{\frac{-\sigma - 1}{1 - \sigma}}t\right)
  & \text{if $\sigma < -1$;}
  \end{cases}
  $$
  with 
  $$
  t = \tan\left( \frac{\theta}{2}\right).$$
Function $G_{\sigma}$ was also computed using AI.
  
\section*{Acknowledgments}
Ernesto Pimentel-García and Carlos Parés have been supported by projects PID2022-137637NB-C21 and PCI2024-155061-2 funded by \\MCIN/AEI/10.13039/501100011033/ and ERDF A way of making Europe. Both have also been supported by the Grant PPRO-FQM216-2023 funded by Consejería de Universidad, Investigación e Innovación and by ERDF Andalusia Program 2021-2027. 
Christophe Berthon and Marie Compain acknowledge the support of the ANR project COOKIE ANR-25-CE40-5565.
Marie Compain has been supported by the CNRS IRP PICASSO.
The authors thank Manuel J. Castro for the helpful discussions about this work.

\bibliographystyle{plain}
\bibliography{references}

\end{document}